\documentclass[english]{smfart}
\usepackage{hertling-sabbah}

\begin{document}

\title{Examples of non-commutative Hodge~structures}

\author[C.~Hertling]{Claus Hertling}
\address{Lehrstuhl für Mathematik VI\\
Universität Mannheim\\
Seminargebäude A~5, 6\\
68131 Mannheim\\
Germany}
\email{hertling@math.uni-mannheim.de}

\author[C.~Sabbah]{Claude Sabbah}
\address{UMR 7640 du CNRS\\
Centre de Mathématiques Laurent Schwartz\\
École polytechnique\\
F--91128 Palaiseau cedex\\
France}
\email{sabbah@math.polytechnique.fr}
\urladdr{http://www.math.polytechnique.fr/~sabbah}

\thanks{This research was supported by the grant ANR-08-BLAN-0317-01 of the Agence nationale de la recherche.}

\begin{abstract}
We show that\addedo{, under a condition called minimality,} if the Stokes matrix of a connection with a pole of order two and no ramification gives rise, when added to its adjoint, to a positive semi-definite Hermitian form, then the associated integrable twistor structure (or TERP structure, or non-commutative Hodge structure) is pure and polarized.
\end{abstract}



\maketitle
\vspace*{-.5cm}
\tableofcontents
\mainmatter

\section*{Introduction}
It is relatively easy to produce examples of variations of polarized Hodge structures on the complement $\Afu\moins C$ of a finite set~$C$ in the complex affine line $\Afu$. The simplest ones consist of variations of type $(0,0)$, that is, flat holomorphic bundles on $\Afu\moins C$ with a flat Hermitian metric, together with a flat real (\resp rational, \resp integral) structure, depending on whether the Hodge structures are real (\resp rational, \resp integral). Equivalently, such variations are in one-to-one correspondence with~$\RR$- (\resp $\QQ$-, \resp $\ZZ$-) local systems on $\Afu\moins C$ whose monodromy representation takes values in the unitary group (up to conjugation). Other classical variations arise whenever one is given a projective morphism $f:X\to\Afu$ on a smooth complex quasi-projective variety $X$ and~$C$ is the set of critical values of $f$, as Gauss-Manin systems of $f$.

Recently, a generalization of the notion of variation of polarized Hodge structures has been considered under the names of variation of integrable polarized twistor structure (generalizing \emph{complex} variations of polarized Hodge structures, \cf \cite{Bibi01c}, \cite{Mochizuki08b}), variations of \replacedo{pure polarized TERP}{tr-TERP} structures (generalizing \emph{real} variations of polarized Hodge structures, \cf \cite{Hertling01,H-S06,H-S08b}), and variations of non-commutative Hodge structures (generalizing \emph{rational} variations of polarized Hodge structures, \cf \cite{K-K-P08}); the case with a $\ZZ$-structure has also been considered in \cite{Iritani09}.

The interest of such generalizations comes from the following observations.
\begin{enumerate}
\item
While variations of polarized Hodge structures degenerate with regular singularities, the previous generalizations may degenerate with irregular singularities, and thus can extend the scope of the theory. In particular, Fourier-Laplace transform\addedo{ation} can be extended to such objects (\cf \cite{Bibi05,Bibi08}) and they form part of the larger family of wild twistor $\cD$-modules (\cf \cite{Mochizuki08} and also \cite{Bibi06b}).
\item
Mirror symmetry produces such structures in quantum cohomology (\cf \cite{C-V91,C-F-I-V92,C-V93,Iritani09b}.
\item
These structures are convenient to adapt the techniques of classical Hodge theory (in particular period mappings) to the local analytic settings attached to isolated singularities of complex hypersurfaces (classifying spaces of Brieskorn lattices, \cf \cite{Hertling99b,Hertling01,Hertling06,H-S06,H-S07,H-S08,H-S08b}).
\end{enumerate}

An integrable twistor structure consists of a germ of holomorphic bundle on a disc with coordinate $\hb$ (say), equipped with
\begin{itemize}
\item
a meromorphic connection having a pole of order at most two at the origin \deletedo{(``without ramification'', see below)} and no other pole,
\item
a nondegenerate \addedn{bi}linear pairing between the underlying local system on $\{\hb
\neq0\}$ and the pull back by $\iota:\hb\mto-\hb$ of its conjugate local system \addedo{which satisfies a skew-Hermitian property} (we call such a pairing a \emph{$\iota$-\replacedo{skew-Hermitian}{sesquilinear} pairing} on the local system).
\end{itemize}
These data allow one to construct in a natural way (\emph{twistor gluing}) a holomorphic vector bundle on $\PP^1$. When this bundle is trivial, we say (\cf \cite{Simpson97}) that the twistor structure is \emph{pure of weight $0$}\deletedo{ and polarized}. \addedo{The construction then equips the space of global sections of this bundle with a nondegenerate Hermitian pairing. If this pairing is positive definite, we say that the pure twistor structure is \emph{polarized}. In the following, ``pure and polarized'' will usually mean ``pure of weight $0$'' and polarized.}

The Riemann-Hilbert correspondence for meromorphic connections with slope one (as only positive slope) and no ramification (that we will call below \emph{of exponential type}, like in \cite{K-K-P08}, \cf \eg\cite[Lemma 1.5]{Bibi08} for the relation with regularity after Laplace transform\addedo{ation}), enriched with such a pairing, allows one to encode the data of the meromorphic bundle and the pairing in a block-upper triangular matrix $\Sigma$ \addedm{(the unipotent Stokes matrix multiplied by the ``square root'' of the formal monodromy)} with invertible diagonal blocks\addedo{ and a set of exponential factors}. It remains to choose, within the meromorphic bundle, a holomorphic bundle on which the connection has a pole of order at most two. If the connection is of exponential type, a canonical holomorphic bundle is provided by the Deligne-Malgrange lattice (with the choice $(0,1]$ for the real part of the eigenvalues of the residues, \cf\S\ref{subsec:BriDe}). Therefore, such a matrix $\Sigma$ also determines the Deligne-Malgrange lattice.

Our main result (Theorem \ref{th:main}) answers Conjecture 10.2 in \cite{H-S06}: \addedo{if an arbitrary set of exponential factors is given and} if $\Sigma$ as above is such that $\Sigma+{}^t\ov\Sigma$ is positive semi-definite\addedn{ and satisfies a property called minimality}\addedm{ (\cf Definition \ref{def:critminext})}, then the integrable twistor structure \replacedo{which they}{it} determine\deletedo{s} (with the Deligne-Malgrange lattice) is pure of weight~$0$ and polarized.\addedm{ In fact, the statement that we give slightly relaxes this minimality property.} Note that if $\Sigma$ is real (\resp rational), the corresponding integrable twistor structure is then a \replacedo{pure polarized TERP}{tr-TERP} structure in the sense of \cite{Hertling01} (\resp a non-commutative Hodge structure in the sense of \cite{K-K-P08}).

The question of how to compute as explicitly as possible the `new supersymmetric index' of Cecotti and Vafa \cite{C-F-I-V92} for such a polarized pure twistor structure remains open (\cf \cite{Hertling01,Bibi01c,Bibi08,Mochizuki08b} for the definition and some properties in the present setting).

The proof of Theorem \ref{th:main} consists in showing that the integrable twistor structure determined by $\Sigma$ is nothing but the twistor structure associated to the Laplace transform of a regular holonomic module with a flat Hermitian form on its smooth part. We essentially identify the restriction of this Hermitian form to the fibre at some general point with the form defined by $\Sigma+{}^t\ov\Sigma$. If it is positive\addedo{ definite}, then the flat bundle has a Hermitian metric, and it follows from \cite{Bibi05} that the twistor structure corresponding to the Fourier-Laplace transform is pure of weight $0$ and polarized.

We use the algebraic/analytic version of the Laplace transform\addedo{ation}, as it is simpler to prove the Fourier inversion formula in this setting. A topological version of the Laplace transform\addedo{ation} (homological with Lefschetz thimbles or cohomological like in \cite{K-K-P08}\replacedo{ and}{) is also available (}including the Stokes structure)\addedo{ also exists}, but \replacedo{we did not find a complete reference for}{at the moment there is no direct proof of} the \addedo{corresponding} Fourier inversion formula \replacedo{in this purely topological setting}{for it}.

\section{Polarized pure twistor structure attached to a flat unitary bundle}
In this section, we will recall some of the results of \cite{Bibi05} in the particular case of a variation of polarized pure Hodge structure of type $(0,0)$ (flat unitary bundle). The consequence of these results, given by Corollary \ref{cor:N}, will be our main tool for proving Theorem \ref{th:main}.

\subsection{Sesquilinear pairings on $\Clt$-modules}\label{subsec:sesqM}
Let~$\Afu_t$ be the complex affine line with coordinate~$t$ and let $C=\{c_1,\dots,c_r\}\subset\Afu_t$ be a finite set of points. We denote by $\Clt$ the Weyl algebra of the variable $t$, by $\cD_{\PP^1}$ the sheaf of holomorphic differential operators on $\PP^1$, and by $\cD_{\PP^1}(*\infty)$ its localization at infinity, so that $\Clt=\Gamma(\PP^1,\cD_{\PP^1}(*\infty))$. Recall that the classical Riemann-Hilbert correspondence gives an equivalence between the following categories \eqref{enum:RH1}--\eqref{enum:RH3}, and an extension of it to $\cD$-modules together with a GAGA argument gives the equivalence with \eqref{enum:RH4} and~\eqref{enum:RH5}:
\begin{enumerate}
\item\label{enum:RH1}
Locally constant sheaves $\cV$ of finite dimensional $\CC$-vector spaces on $\Afu_t\moins C$ (that we call local systems for short),
\item\label{enum:RH2}
holomorphic flat bundles with connection $(V,\nabla)$ on $\Afu_t\moins C$,
\item\label{enum:RH3}
locally free $\cO_{\PP^1}(*C\cup\{\infty\})$-modules $\wt\cM$ with regular singular connection,
\item\label{enum:RH4}
regular holonomic $\cD_{\PP^1}$-modules $\cM$ with singularities at $C\cup\{\infty\}$, which are minimal extension\addedo{s} at~$C$ (\ie have neither sub nor quotient module supported on~$C$) and maximal extension\addedo{s} at $\infty$ (\ie are $\cD_{\PP^1}(*\infty)$-modules),
\item\label{enum:RH5}
regular holonomic $\Clt$-modules~$M$ with singularities at~$C$ and \addedo{which} have neither sub nor quotient module\addedo{s} supported on~$C$.
\end{enumerate}

This correspondence extends to a correspondence with sesquilinear pairing as follows. Let $\cS'(\Afu_t)$ be the Schwartz space of tempered distributions on~$\Afu_t$. This is the space of global sections of the sheaf $\Db_{\PP^1_t}^{\rmod\infty_t}$ on $\PP^1_t$ of distributions on~$\Afu_t$ which have moderate growth at infinity (\addedo{on any open set $U$ of $\PP^1_t$, its space of sections is the dual of the space of $C^\infty$ functions with compact support on $U$ having rapid decay at infinity;} it can be regarded as the quotient of the sheaf of distributions on~$\PP^1_t$ modulo distributions supported at infinity \addedo{and is also equal to the localized sheaf $\cO_{\PP^1}(*\infty)\otimes_{\cO_{\PP^1}}\Db_{\PP^1_t}$, according to the division property of distributions by holomorphic functions}). We will also consider the sheaf $\Db_{\PP^1_t}^{\rmod C\cup\infty_t}$ on $\PP^1_t$ of distributions on $\Afu_t\moins C$ having moderate growth at $C\cup\{\infty_t\}$.

Then, any sesquilinear pairing $h\addedn{_\rB}:\cV'\otimes_\CC\ov\cV{}''\to\CC_{\Afu_t\moins C}$ between the local systems~ $\cV'$ and $\cV''$ (where $\ov\cV{}''$ denotes the conjugate local system and ``sesquilinear'' means that~$h\addedn{_\rB}$ is a $\CC$-linear morphism) induces in a unique way a sesquilinear pairing\addedn{~$h$} on the minimal extensions taking values in the Schwartz space of tempered distributions on~$\Afu_t$ and which is linear with respect to the natural $\Clt\otimes_\CC\ov{\Clt}$-action on both the source and the target. Indeed, it is easy to extend~$h\addedn{_\rB}$ as a~$\nabla$-flat sesquilinear pairing $\addedn{h:}V'\otimes_\CC\ov V{}''\to\cC^\infty_{\Afu_t\moins C}$, \ie which satisfies
\[
h(\nabla v',\ov{v''})=\partial h(v',\ov{v''})\quad\text{and}\quad h(v',\ov{\nabla v''})=\ov\partial h(v',\ov{v''})
\]
for all local sections $v',v''$ of $V',V''$. Since any local meromorphic basis of $\wt\cM$ can be expressed with coefficients having moderate growth in any basis of local horizontal sections (according to the regularity of the connection), the pairing extends as a sesquilinear pairing between $\wt\cM{}'$ and the conjugate of $\wt\cM{}''$ taking values in the sheaf of distributions on~$\PP^1_t$ having moderate growth at $C\cup\{\infty_t\}$ (sesquilinearity means $\cD_{\PP^1_t}(*\infty_t)\otimes\ov{\cD_{\PP^1_t}(*\infty_t)}$-linearity). The latter induces such a pairing between the minimal extensions $\cM',\ov\cM{}''$ with values in $\Db_{\PP^1_t}^{\rmod C\cup\infty_t}$. A local inspection of the values of this pairing near the points of~$C$ shows that it can be canonically lifted as a pairing with values in the sheaf on $\PP^1_t$ of distributions on \replacedo{$\Afu_t$}{$\PP^1_t$} with moderate growth at infinity\deletedo{ only}. Taking global sections on~$\PP^1_t$ gives a sesquilinear pairing $\addedn{h:}M'\otimes\ov M{}''\to\cS'(\Afu_t)$. Going in the other direction from~$M$ to $\cV$ is easier\addedn{ via the de~Rham functor, and we denote by $h_{\DR}$ the corresponding form}.

Let $h:M'\otimes_\CC\ov M{}''\to\cS'(\Afu_t)$ be a sesquilinear pairing between holonomic $\Clt$-modules. We can view~$h$ as a $\Clt$-linear morphism $M'\to\Hom_{\ov{\Clt}}(\ov M{}'',\cS'(\Afu_t))$, where the latter module is equipped with the $\Clt$-module structure coming from that on $\cS'(\Afu_t)$. It is known (but not used now) that $\Hom_{\ov{\Clt}}(\ov M{}'',\cS'(\Afu_t))$ is also a holonomic $\Clt$-module (\cf \cite[Cor\ptbl II.3.4.2]{Bibi97}). We will say that~$h$ is \emph{nondegenerate} if it induces an isomorphism $M'\isom\Hom_{\ov{\Clt}}(\ov M{}'',\cS'(\Afu_t))$. If $M'=M''$, the definition of ``Hermitian'' is the obvious one.

Similarly, one can define the notion of ``nondegenerate'' and ``Hermitian'' at all steps of the Riemann-Hilbert correspondence above. It is easy to see that if $h:M'\otimes_\CC\nobreak\ov M{}''\to\cS'(\Afu_t)$ is nondegenerate, then so is $h\addedn{_{\DR}}:\cV'\otimes\ov\cV{}''\addedo{\to \addedn{\CC}_{\Afu_t\moins C}}$ (by sheafifying the morphism $M'\to\Hom_{\ov{\Clt}}(\ov M{}'',\cS'(\Afu_t))$ and restricting to $\Afu_t\moins C$). The converse also holds, but we will not need it in this article (in a special case the result follows from Lemma \ref{lem:RHh} below).

Let us also notice that, if a $\Clt$-module~$M$ has regular singularities at $C\cup\{\infty\}$ and is equipped with a nondegenerate sesquilinear pairing~$h$, then~$M$ is a minimal extension at its singularity set~$C$ if and only if~$M$ has no submodule supported by~$C$ (a quotient module supported by~$C$ would produce a submodule of $\Hom_{\ov{\Clt}}(\ov M,\cS'(\Afu_t))\simeq M$ supported by~$C$).

Lastly, we remark that if~$h$ is Hermitian and nondegenerate on~$M$, it is so on~$V$\addedo{ (and the connection on $V$ is the holomorphic part of \addedn{the} Chern connection of~$h$)}, and then it is positive definite at one fibre of $V$ if and only if it is so at any fibre of~$V$ (because $\Afu_t\moins C$ is connected). In such a case, $V$ is a holomorphic vector bundle on $\Afu_t\moins C$ with a flat Hermitian metric~$h$\deletedo{ (so that the connection on $V$ is the holomorphic part of Chern connection of~$h$)}. By the Riemann-Hilbert correspondence (taking horizontal sections), it corresponds to a locally constant sheaf $\cV$ of complex vector spaces on $\Afu_t\moins C$ whose monodromy is unitary, that is, whose associated monodromy representation takes values, up to conjugation, in the unitary group. In particular the representation is semi-simple and, going back through the Riemann-Hilbert correspondence, the corresponding $\Clt$-module~$M$ is semi-simple.

\subsection{Laplace transform\addedo{ation} and sesquilinear pairings}\label{subsec:LaplaceD}
Let~$M$ be a holonomic $\Clt$-module and let $N=\Fou M$ be its Laplace transform with kernel $e^{-t\tau}$: by definition,~$\Fou M$ coincides with~$M$ as a $\CC$-vector space and the $\Cltau$ action is defined by $\tau\cdot m=\partial_tm$, $\partial_\tau m=-tm$. It is known (\cf \eg \cite[Chap\ptbl V]{Malgrange91}) that Laplace transform\addedo{ation} gives a one-to-one correspondence between regular holonomic $\Clt$-modules and holonomic $\Cltau$-modules with a regular singularity at $\tau=0$ and an irregular one of exponential type at infinity, in the following sense. Let us set $G=\CC[\tau,\tau^{-1}]\otimes_{\CC[\tau]}\Fou M$ and $\hb=\tau^{-1}$. This is a free $\CC[\tau,\tau^{-1}]$-module of finite rank equipped with a connection. Then $\wh G\defin\CC\lcr\hb\rcr\otimes_{\CC[\hb]}G$ is a free $\CC\lpr\hb\rpr$-vector space with connection, isomorphic to
\[
\tbigoplus_{c\in C}(\cE^{\addedo{-}c/\hb}\otimes\wh R_c)
\]
(called the \emph{Levelt-Turrittin decomposition}), where $R_c$ has a regular connection and $\cE^{\addedo{-}c\addedn{/\hb}}\defin(\CC\lcr\hb\rcr,\replacedo{d-cd(1/\hb)}{})$.

We will denote by~$\Afu_t$ (\resp $\Afu_\tau$) the affine line with coordinate $t$ (\resp $\tau$) and by
\[
F_t:\cS'(\Afu_t)\to\cS'(\Afu_\tau)
\]
the Fourier transform\addedo{ation} of tempered distributions with kernel $e^{\ov{t\tau}-t\tau}\itwopi dt\wedge d\ov t$. \addedo{Recall that, given a function $\chi(\tau)$ in the Schwartz space $\cS(\Afu_\tau)$ (\ie $\chi(\tau)$ $C^\infty$, rapidly decaying as well as all its derivatives when $\tau\to\infty$), we set $\psi=\chi(\tau)d\tau\wedge d\ov\tau$ and, for $T\in\cS'(\Afu_t)$,
\[
\langle F_t\addedo{T},\psi(\tau)\rangle\defin\langle \addedo{T},(F_\tau\psi)\itwopi dt\wedge d\ov t\rangle,\quad \text{with }(F_\tau\psi)(t)=\int_{\Afu_\tau} e^{\ov{t\tau}-t\tau}\psi(\tau) \in \cS(\Afu_t).
\]
}

\addedo{The Fourier transform $F_t$ is an isomorphism between $\cS'(\Afu_t)$ and $\cS'(\Afu_\tau)$. Moreover, defining similarly $F_\tau:\cS'(\Afu_\tau)\to\cS'(\Afu_t)$ with the kernel $e^{\ov{t\tau}-t\tau}\itwopi d\tau\wedge d\ov\tau$, we have
\[
F_t^{-1}=\ov F_\tau
\]
(where $\ov F_\tau$ has kernel $e^{t\tau-\ov{t\tau}}\itwopi d\tau\wedge d\ov\tau$). Indeed, it is enough to check the dual relation for the Fourier transform of functions in the Schwartz class\addedn{es} $\cS(\Afu_t)$ and $\cS(\Afu_\tau)$. Let us set $t=(x+iy)/\sqrt2$ and $\tau=(\xi+i\eta)/\sqrt2$. Let $\varphi=\chi(x,y)dx\wedge dy$ \replacedn{with $\chi$}{be} in the Schwartz class on $\Afu_t$. If we set $s=(u+iv)/\addedn{\sqrt2}$, the assertion amounts to
\[
\bigg[\int_{\Afu_\tau}e^{\ov{t\tau}-t\tau}\Big(\int_{\Afu_t}e^{s\tau-\ov{s\tau}}\varphi(u,v)\Big)\itwopi d\tau\wedge d\ov\tau\bigg]\itwopi dt\wedge d\ov t=\varphi(x,y),
\]
or equivalently
\[
\frac{1}{4\pi^2}\int_{\Afu_\tau}e^{-i(x\eta+y\xi)}\Big(\int_{\Afu_t}e^{i(u\eta+v\xi)}\chi(u,v)du\wedge dv\Big)d\xi\wedge d\eta=\chi(x,y).
\]
Here, $\Afu$ is oriented with its complex\deletedn{ed} structure, so that if we denote by $du\cdot dv$ the Lebesgue measure and $|\Afu|=\RR^2$ without orientation, we have $\int_{\Afu_t}\cbbullet\ du\wedge dv=\int_{|\Afu_t|}\cbbullet\ du\cdot dv$, so our assertion reduces to the standard Fourier inversion formula for functions in the Schwartz class of $\RR^2$.
}

It is well-known that $F_t$ and $\ov F_\tau$ are linear with respect to the $\Clt\otimes_\CC\ov{\Clt}$-action on $\cS'(\Afu_t)$ and the $\Cltau\otimes_\CC\ov{\Cltau}$-action on $\cS'(\Afu_\tau)$ via the correspondence $\partial_t\leftrightarrow\tau$, $t\leftrightarrow-\partial_\tau$ defined above.

If $h:M'\otimes_\CC\ov M{}''\to\cS'(\Afu_t)$ is a sesquilinear pairing, we define the Fourier transform~$\Fou h$ as the composition $F_t\circ h$. \replacedo{In order to interpret $\Fou h$ as a $\Cltau\otimes_\CC\ov{\Cltau}$-linear morphism, and thus to keep sesquilinearity, we have to use the kernels $e^{-t\tau}$ on~$M'$ and $e^{\ov{t\tau}}$ on~$\ov M{}''$, that is, to regard $\Fou h$ as a pairing from the Laplace transform of~$M'$ and the conjugate of the \emph{inverse} Laplace transform of $M''$, which is nothing but $\iota^+\ov N{}''$, if we denote by~$\iota$ the involution $\tau\mto-\tau$. We therefore view $\Fou h$ as sesquilinear pairing $N'\otimes_\CC\iota^+\ov N{}''\to\cS'(\Afu_\tau)$}{This is a sesquilinear pairing $N'\otimes_\CC\iota^+\ov N{}''\to\cS'(\Afu_\tau)$, where~$\iota$ denotes the involution $\tau\mto-\tau$, because $\Fou\ov M=\iota^+\ov{\Fou M}$}. Note that we recover~$h$ as $\ov F_\tau F_t h$. We will also set $\varh\defin\itwopi\Fou h$. It is important to notice that $\varh$ (or~$\Fou h$) is nondegenerate if and only if~$h$ is so. This follows from the fact that $F_t$ is an isomorphism.

We also remark that, if $M'=M''=M$,~$h$ is Hermitian if and only if~$\Fou h$ is $\iota$\nobreakdash-Hermitian. Indeed, $\iota$ induces an involution $\iota^*:\cS'(\Afu_\tau)\isom\cS'(\Afu_\tau)$ and we have $F_t(\ov{\addedo{T}})=\iota^*\ov{F_t\addedo{T}}$ for $\addedo{T}\in\cS'(\Afu_t)$. This is also equivalent to $\varh=\itwopi\Fou h$ being $\iota$-skew-Hermitian\addedm{ (the choice of the sign $+i$ is irrelevant here, it will be justified by the comparison lemma \ref{lem:hBc})}.

\subsection{A criterion on~$\Fou M$ asserting that~$M$ is a minimal extension}
Let~$M$ be a regular holonomic $\Clt$-module with singularities at~$C$\deletedn{, equipped with a nondegenerate sesquilinear pairing~$h$}. Let $N=\Fou M$ be its Laplace transform and set $G=\CC[\tau,\tau^{-1}]\otimes_{\CC[\tau]}N$ as above.

\begin{lemme}\label{lem:critminext}
\addedn{Let us assume that
\begin{itemize}
\item
$M$ is equipped with a nondegenerate sesquilinear pairing~$h$,
\item
$N=\Fou M$ is a minimal extension at $\tau=0$ (its regular singularity).
\end{itemize}}
\replacedn{Then}{With these assumptions,}~$M$ is a minimal extension if and only if $G$  has no rank-one $\CC[\tau,\tau^{-1}]$-submodule stable by~$\nabla$ on which the monodromy is the identity.
\end{lemme}

\begin{proof}
Because of the existence of~$h$,~$M$ is a minimal extension if and only if it has no submodule isomorphic to $\Clt/\Clt(t-c)$, with $c\in C$. This is equivalent to asking that~$N$ has no submodule isomorphic to $(\CC[\tau],\addedo{d-cd\tau})$.\addedn{ On the other hand, any $\CC[\tau,\tau^{-1}]$-submodule of $G$ stable by~$\nabla$ has a regular singularity at the origin and has exponential type at infinity. If such a module has rank one and if the monodromy is the identity, it must be equal to $(\CC[\tau,\tau^{-1}],d-cd\tau)$ for some $c\in \CC$ (in fact some $c\in C$).}

\addedn{Assume that $M$ is a minimal extension. If we had a submodule $(\CC[\tau,\tau^{-1}],d-cd\tau)$ in $G$, then $(\CC[\tau],d-cd\tau)$ would be a $\Cltau$-submodule of $G$. Since $N$ is a minimal extension at $\tau=0$, it is included in $G$. Since the intersection in $G$ of $(\CC[\tau],d-cd\tau)$ and~$N$ is non-zero (because it is non-zero after localization), and since $(\CC[\tau],d-cd\tau)$ is a simple $\Cltau$-module, $(\CC[\tau],d-cd\tau)$ would be contained $N$. By inverse Laplace transform, $M$ would have a submodule supported on $C$, a contradiction.}

\replacedn{Conversely, assume that $G$ is as in the lemma. Then $N$ does not have any $\Cltau$-submodule isomorphic to $(\CC[\tau],d-cd\tau)$ (otherwise, by localization, it would produce a $(\CC[\tau,\tau^{-1}],d-cd\tau)$ in~$G$). By inverse Laplace transform, $M$ has no sub-module supported on $C$.}{Because of the existence of~$h$,~$M$ is a minimal extension if and only if it has no submodule isomorphic to $\Clt/\Clt(t-c)$, with $c\in C$. This is equivalent to asking that~$N$ has no submodule isomorphic to $(\CC[\tau],\addedo{d-cd\tau})$, and it can be checked that this is equivalent to asking that $G$ has no $\CC[\tau,\tau^{-1}]$-submodule stable by~$\nabla$ isomorphic to $(\CC[\tau,\tau^{-1}],\addedo{d-cd\tau})$.}
\end{proof}

\begin{remarque}
In particular, if we assume that $1$ is not an eigenvalue of the monodromy on $(G,\nabla)$, then the condition of the lemma is fulfilled and~$M$ is a minimal extension.

\deletedo{The condition that~$N$ is a minimal extension at $\tau=0$ also reads on~$M$: assuming the existence of a nondegenerate~$h$ on~$M$ (equivalently, of $\varh$ on~$N$), there should be no $m\in M$ such that $\partial_tm=0$ (equivalently, $(\CC[t],d)$ is not a $\Clt$-submodule of~$M$).}
\end{remarque}

\subsection{The Brieskorn lattice of a Deligne lattice}\label{subsec:BriDe}
Let~$M$ be a regular holonomic $\Clt$-module with singularities at~$C$. Assume that~$M$ is a minimal extension at~$C$. Assume also that the eigenvalues of the local monodromies of the corresponding local system $\cV$ have absolute value equal to one (this property holds if $\cV$ is unitary). Let us denote by $V^{>-1}M$ the free $\CC[t]$-submodule of~$M$ satisfying the following two properties:
\begin{enumerate}
\item
the connection~$\nabla$ on~$M$ induces a logarithmic connection on $V^{>-1}M$,
\item
the eigenvalues of the residues at~$C$ (which are real by the assumption on the local monodromies) belong to $(-1,0]$.
\end{enumerate}
Because~$M$ is assumed to be a minimal extension, it is generated, as a $\Clt$-module, by $V^{>-1}M$.

The \emph{Brieskorn lattice} $G_0$ attached to $V^{>-1}M$ is, by definition, the $\CC[\tau^{-1}]$-submodule of $G$ generated by the image of $V^{>-1}M$ in $G$ via the localization morphism $M=\Fou M\to G$.

Each $\wh R_c$ in the Levelt-Turrittin decomposition of $\wh G$ has a formal Deligne lattice $V^{>0}\wh R_c$ which is the unique logarithmic lattice for which the eigenvalues of the residue of the connection belong to $(0,1]$, and therefore (according \eg to \cite[Prop\ptbl2.1]{Malgrange04}) $G$ has a unique $\CC[\hb]$-lattice $\addedm{\DM}^{>0}G$ whose associated formal lattice is $\bigoplus_{c\in C}(\cE^{\addedo{-}c\addedn{/\hb}}\otimes\nobreak V^{>0}\wh R_c)$. We call $\addedm{\DM}^{>0}G$ the \emph{Deligne-Malgrange lattice} of $G$ at $\infty$.

\begin{lemme}\label{lem:BDM}
We have $G_0=\addedm{\DM}^{>0}G$.
\end{lemme}

\begin{proof}
It is known that $\CC\lcr\hb\rcr\otimes_{\CC[\hb]}G_0$ decomposes as $\bigoplus_{c\in C}(\cE^{\addedo{-}c\addedn{/\hb}}\otimes (V^{>-1}M)^\mu_c)$, where $(V^{>-1}M)^\mu_c$ is the formal microlocalization of $V^{>-1}M$ at~$c$ (\cf \cite[Prop\ptbl2.3]{Bibi96bb}). The identification of $(V^{>-1}M)^\mu_c$ with $V^{>0}\wh R_c$ is then standard.
\end{proof}

\subsection{Twistor gluing}\label{subsec:tw}
Let $(\cH,\nabla)$ be a free $\CC\{\hb\}$-module of finite rank with a meromorphic connection having a pole of order $\leq2$ at $\hb=0$. We will assume that the connection~$\nabla$ on the associated meromorphic bundle $\cH(*0)=\CC(\{\hb\})\otimes_{\CC\{\hb\}}\cH$ is of \emph{exponential type}. Let $\cH^\nabla$ denote the local system of horizontal sections of~$\nabla$ on a small punctured disc $\Delta^*$ centered at $\hb=0$. Assume moreover that we are given a nondegenerate $\iota$-skew-Hermitian pairing $\varh_\rB:\cH^\nabla\otimes_\CC\iota^{-1}\ov{\cH^\nabla}\to\CC_{\Delta^*}$, where $\iota$ is the involution $\hb\mto-\hb$\addedo{ (the index $\rB$ is for ``Betti'', as such a pairing is often defined in a topological way)}. \addedm{We associate to $\varh_\rB$ the $\iota$-Hermitian pairing $-2\pi i\varh_\rB$.}

Using the flat connection~$\nabla$, it is possible to extend in a unique way the previous objects as analogous objects on the complex line $\Afuan_\hb$. On the circle $|\hb|=1$, the involution $\iota$ coincides with the anti-linear involution $\sigma:\hb\mto-1/\ov\hb$, and $\addedm{-2\pi i}\varh_{\rB|S^1}$ can be used to glue $\cH^\vee$ (dual of $\cH$) with $\sigma^*\ov\cH$, to get a holomorphic bundle on~$\PP^1$, that is (as this is compatible with the connection) an \emph{integrable twistor structure}. We say that this twistor structure is obtained by twistor gluing of $(\cH,\nabla,\varh\addedo{_\rB})$ (\cf \addedn{\cite[Lemma 2.14]{Hertling01}}, \cite[Def\ptbl1.29]{Bibi05}).

An example where the resulting twistor structure is pure of weight $0$ (or of some weight) and polarized is obtained as follows (\cf\cite{Bibi05}). Let~$M$ be a regular holonomic $\Clt$-module, which is a minimal extension at its singular set~$C$, and which is endowed with a Hermitian pairing~$h$ with values in $\cS'(\Afu_t)$. Then its Laplace transform~$\Fou M$ is a holonomic $\Cltau$-module, with a regular singularity at $\tau=0$ and an irregular one of exponential type at $\tau=\infty$. The Fourier transformed pairing~$\Fou h$ induces a $\iota$-skew-Hermitian pairing \replacedm{$\varh_\rB=\itwopi\Fou h$}{$(\Fou h)_\rB$} on the corresponding local system $\Fou\cV$. On the other hand, we denote by $\addedm{\DM}^{>0}G$ the Deligne\added{-Malgrange} lattice of~$G=\CC[\tau,\tau^{-1}]\otimes\Fou M$ at $\tau=\infty$ (which is also the Brieskorn lattice $G_0$ of the Deligne lattice $V^{>-1}M$\addedm{, according to Lemma \ref{lem:BDM}}). Setting $\hb=\tau^{-1}$, we can thus apply the twistor gluing procedure to these data. As a direct consequence of \cite[Cor\ptbl3.15]{Bibi05} (using that a flat Hermitian bundle is a variation of complex Hodge structures of type $(0,0)$), we get:

\begin{proposition}\label{prop:polar}
If the pairing~$h$ is Hermitian positive definite on $(V,\nabla)$, then the integrable twistor structure attached to \replacedm{$(\DM^{>0}G,\nabla,\Fou h=-2\pi i\varh_\rB)$}{the previous data} is pure of weight~$0$ and polarized.\qed
\end{proposition}

As a consequence of the previous results we obtain:

\begin{corollaire}\label{cor:N}
Let~$N$ be a holonomic $\Cltau$-module of exponential type at infinity \deletedo{which is a minimal extension at its} \addedo{having a} single singularity \addedo{at} $0$ in $\Afu_\tau$\addedo{, which is regular}. Set $G=\CC[\tau,\tau^{-1}]\otimes_{\CC[\tau]}N$. Assume that
\begin{enumerate}
\item\label{cor:N1}
\replacedo{$G$ has no rank-one $\CC[\tau,\tau^{-1}]$-submodule stable by~$\nabla$ on which the monodromy is the identity}{\ $N$ fulfills the condition of Lemma \ref{lem:critminext}},
\addedn{\item\label{cor:N1b}
$N$ is a minimal extension at $\tau=0$,}
\item\label{cor:N2}
$N$ is equipped with a \addedo{nondegenerate} $\iota$-skew-Hermitian pairing $\varh$ such that $\wh\varh\defin-2\pi i\ov F_\tau\varh$ is positive definite at one fibre $c\not\in C$ (set of exponential factors of~$N$ at infinity).
\end{enumerate}
Then the triple $(\addedm{\DM}^{>0}G,\nabla,-2\pi i\varh)$ defines, by twistor gluing, an integrable twistor structure which is pure of weight $0$ and polarized.
\end{corollaire}

\begin{proof}
The \deletedo{preliminary} assumption \addedo{that $N$ has $0$ as its single singularity, which is regular, at finite distance and has an irregular singularity of exponential type at infinity} means that $N=\Fou M$ for some regular holonomic~$M$ (\cf \cite[Chap\ptbl V]{Malgrange91} or \cite[Lemma 1.5]{Bibi08}). \addedo{Since~$\varh$ is nondegenerate on $N$ by \ref{cor:N}\eqref{cor:N2}, so is~$\wh\varh$ on~$M$ and, according to Lemma \ref{lem:critminext}, \ref{cor:N}\eqref{cor:N1}\addedn{ and \ref{cor:N}\eqref{cor:N1b}},~$M$} is a minimal extension at its singularity set~$C$. \addedo{Moreover,~$\wh\varh$ restricts as a nondegenerate Hermitian form on $(V,\nabla)$. Being positive definite at some $c\notin C$ by \ref{cor:N}\eqref{cor:N2}, it is positive definite all over $\Afu_t\moins C$, and thus} the assumption of Proposition \ref{prop:polar} is satisfied by~$M$\deletedn{, because~$M$ is a regular holonomic $\Clt$-module which is a minimal extension at its singularity set, it is equipped with a nondegenerate Hermitian pairing~$\wh\varh$ with values in $\cS'(\Afu_t)$, and this pairing is positive definite on $(V,\nabla)$}. \addedm{Lastly, we have $\Fou\,\wh\varh=-2\pi i\varh$.}
\end{proof}

\section{Stokes filtration and Stokes data}\label{sec:Stokesfil}
In this section we recall the notion of Stokes filtration as defined in \cite{Deligne78} (\cf also \cite{Malgrange83bb}, \cite{B-V89}, \cite{Malgrange91}) in the particular case of Stokes filtrations which are of exponential type. We make explicit the correspondence to the more classical approach via Stokes data, and we mainly focus on the behaviour with respect to a sesquilinear pairing (hence also to duality).

\subsection{Stokes filtration}\label{subsec:Stokesfil}
Let $\kk$ be a field (\eg $\QQ$ or $\CC$). Let~$\cL$ be a local system of finite dimensional $\kk$-vector spaces on the circle $S^1$ with coordinate $e^{i\theta}$. A Stokes filtration of~$\cL$ is a family of subsheaves $\cL_{\leq c}\subset \cL$, with $c\in\CC$, satisfying the following properties:
\begin{enumerate}
\item\label{enum:Stokesfil1}
For each $\theta\in \replacedo{\RR/2\pi\ZZ}{}$, let $\leqtheta$ be the partial order on $\CC$ which is compatible with addition and satisfies
\[
c\leqtheta0\iff c=0\text{ or } \arg c-\theta\in(\pi/2,3\pi/2)\mod2\pi.
\]
We also set $c\letheta0$ iff $c\neq0$ and $c\leqtheta0$. One requires that, for each $\theta$, the germs $\cL_{\leq c,\theta}$ form an exhaustive increasing filtration of $\cL_\theta$ with respect to $\leqtheta$.
\item\label{enum:Stokesfil2}
Because the order $\leqtheta$ is open with respect to $\theta$, the germs $\cL_{\letheta c}\defin\sum_{c'\letheta c}\cL_{\leq c',\theta}$ glue as a subsheaf $\cL_{<c}$ of~$\cL$. One requires that the graded sheaves $\gr_c\cL\defin\cL_{\leq c}/\cL_{<c}$ are locally constant sheaves on $S^1$.
\item\label{enum:Stokesfil3}
Near any $e^{i\theta}\in S^1$, one requires that there are local isomorphisms $(\cL,\cL_\bbullet)\simeq(\gr\cL,(\gr\cL)_\bbullet)$, where the Stokes filtration on $\gr\cL\defin\bigoplus_{c\in\CC}\gr_c\cL$ is the natural one, that is, $(\gr\cL)_{\leq c,\theta}=\bigoplus_{c'\leqtheta c}\gr_{c'}\cL$. In particular, $\gr_c\cL=0$ except for~$c$ in a finite set $C\subset\CC$, called the set of exponential factors of the Stokes filtration $(\cL,\cL_\bbullet)$.
\end{enumerate}

\begin{remarques}\label{rem:Stokesdir}\mbox{}
\begin{enumerate}
\item\label{rem:Stokesdir1}
We simplify here the general definition of a Stokes filtration, as we only deal with this kind of filtrations. It is called ``of exponential type'' in \cite{K-K-P08}. \addedo{The case where $C=\{0\}$ corresponds to a regular singularity in the setting of bundles with meromorphic connections. One can notice that, as a consequence of the definition, the set~$C$ is not empty except possibly if $\cL=0$; in such a case, it will be convenient to assume also $C\neq\emptyset$, \eg $C=\{0\}$.}

\item\label{rem:Stokesdir2}
For each pair $c\neq c'\in\CC$, there are exactly two values of $\theta\bmod2\pi$, say $\theta_{c,c'}$ and $\theta'_{c,c'}$, such that~$c$ and $c'$ are not comparable at $\theta$. We have $\theta'_{c,c'}=\theta_{c,c'}+\pi$. These values are called the Stokes directions of the pair $(c,c')$. For any $\theta$ in one component of $\replacedo{\RR/2\pi\ZZ}{}\moins\{\theta_{c,c'},\theta'_{c,c'}\}$, we have $c\letheta c'$, and the reverse inequality for any $\theta$ in the other component. We denote \replacedo{the images of these intervals in $S^1$ via $\theta\mto e^{i\theta}$}{\ these intervals} by $S^1_{c\leq c'}$ and $S^1_{c'\leq c}$ respectively. If $c=c'$, we set $S^1_{c\leq c}\defin S^1$.
\item\label{rem:Stokesdir3}
For each pair $c,c_o\in\CC$, the inclusion $j_{c\leq c_o}:S^1_{c\leq c_o}\hto S^1$ is open. We will denote by $\beta_{c\leq c_o}$ the functor $j_{c\leq c_o,!}j_{c\leq c_o}^{-1}$, consisting in restricting a sheaf to this open set and extending the restriction to $S^1$ by $0$. The filtration condition \eqref{enum:Stokesfil1} above implies that, for each pair $c,c_o$, there is a natural inclusion $\beta_{c\leq c_o}\cL_{\leq c}\hto\cL_{\leq c_o}$.
\end{enumerate}
\end{remarques}

A morphism $\lambda:(\cL,\cL_\bbullet)\to(\cL',\cL'_\bbullet)$ of Stokes-filtered local systems is a morphism of local systems satisfying $\lambda(\cL_{\leq c})\subset\cL'_{\leq c}$ for each $c\in\CC$.

\begin{proposition}\label{prop:strict}\mbox{}
\begin{enumerate}
\item\label{prop:strict1}
On any open interval $I\subset \addedo{\RR/2\pi\ZZ}$ of length $\pi+\addedo{2}\epsilon$ with $\epsilon>0$ small, there exists a \addedo{unique} splitting $\cL_{|I}\simeq\bigoplus_c\gr_c\cL_{|I}$ compatible with the Stokes filtrations.
\item\label{prop:strict2}
Let $\lambda:(\cL,\cL_\bbullet)\to(\cL',\cL'_\bbullet)$ be a morphism of Stokes-filtered local systems. Then, for any open interval $I\subset \addedo{\RR/2\pi\ZZ}$ of length $\pi+\addedo{2}\epsilon$, the morphism $\lambda\addedo{_{|I}}$ is graded with respect to the splittings in \eqref{prop:strict1}.
\end{enumerate}
\end{proposition}

\begin{proof}\mbox{}

\eqref{prop:strict1}
\addedo{This is a particular case of \cite[\S5]{Malgrange83bb}.}

\medskip
\eqref{prop:strict2}
By the first part of the \replacedo{proposition}{proof}, choosing a splitting of the Stokes filtration\addedo{s} of~$\cL$ \addedo{and $\cL'$} on~$I$ allows us to decompose $\lambda\addedo{_{|I}}$ into blocks $\lambda_{ij}:\gr_{c_i}\cL\to\gr_{c_j}\cL'$. Each $\lambda_{ij}$ is a morphism of local systems. In particular, it vanishes identically if and only if it vanishes at one point. By assumption, the interval~$I$ contains one (and exactly one) Stokes direction for each pair $(c_i,c_j)$ with $i\neq j$, which is a~$\theta_o$ such that~$c_i$ and~$c_j$ are not comparable with respect to $\leqthetao$. Then, for $\theta$ on one side of~$\theta_o$, one has $c_i\letheta c_j$ and, for $\theta$ on the other side, one has the reverse inequality. Since $\lambda$ is compatible with the Stokes filtration, this implies that $\lambda_{ij}$ ($i\neq j$) vanishes on some nonempty subset of~$I$, and therefore all over~$I$.
\end{proof}
\addedo{\begin{remarque}\label{rem:splitting}
One can regard this splitting result in various ways:
\begin{enumerate}
\item\label{rem:splitting1}
For $\theta,\theta'=\theta+\pi\in I$, the filtrations $\cL_{\leq\cbbullet,\theta}$ and $\cL_{\leq\cbbullet,\theta'}$ are opposite, if one identifies the opposite fibres $\cL_{\theta}$ and $\cL_{\theta'}$ by the flat structure along the interval $I$. The given splitting is the unique common splitting of these opposite filtrations.
\item\label{rem:splitting2}
\deletedn{Using the functor $\beta$ introduced above, the compatibility of the splitting with the Stokes filtration means that, for any $c_o\in\CC$, we have an induced splitting
$\cL_{\leq c_o|I}=\tbigoplus_{c\in C}\beta_{c\leq c_o}\gr_c\cL_{|I}$.
In particular, for $c_o\in C$, the piece $\gr_{c_o}\cL_{|I}$ of the splitting of $\cL$ is the constant sheaf with fibre $\Gamma(I,\cL_{\leq c_o})$. These spaces of global sections fit together to a direct sum which generates all germs of sections of $\cL$ on~$I$.}
\addedn{The pieces of the unique splitting of $\cL_{|I}$ are the constant sheaves $\Gamma(I,\cL_{\leq c_i})$. Proposition \ref{prop:strict}\eqref{prop:strict1} says that these spaces of sections on~$I$ fit together to a direct sum which generates all sections of~$\cL$ on~$I$.}
\end{enumerate}
\end{remarque}}

\begin{proposition}\label{prop:catStokesfiltabelian}
The category of Stokes-filtered local systems $(\cL,\cL_\bbullet)$ is abelian.
\end{proposition}

\begin{proof}
Let $\lambda:(\cL',\cL'_\bbullet)\to(\cL,\cL_\bbullet)$ be a morphism of Stokes structures. Firstly, $\ker\lambda$ and $\coker\lambda$ are local systems on $S^1$. Moreover, on any open interval~$I$ of length $\pi+\addedo{2}\epsilon$, $\lambda$ is graded, according to \ref{prop:strict}\eqref{prop:strict2}. This easily implies that, on each such~$I$, the kernel and the cokernel of $\lambda:\cL'_\bbullet\to\cL_\bbullet$ are Stokes filtrations of $\ker\lambda$ and $\coker\lambda$ respectively, so that $\ker\lambda$ and $\coker\lambda$ exist as Stokes-filtered local systems, and the morphism of the co-image to the image of $\lambda$ is an isomorphism, so the category is abelian.\end{proof}

Let $\Afu_\tau$ be the affine line with coordinate $\tau$. From \cite{Deligne78} (\cf also \cite{Malgrange83bb}, \cite{B-V89}, \cite{Malgrange91}), we get:

\begin{proposition}\label{prop:equivconnStokesloc}
If $\kk=\CC$, there is an equivalence between the category of \replacedo{rational}{meromorphic} connections on $\Afu_\tau$ with a regular singularity at $\tau=0$ and of exponential type at $\tau=\infty$, and the category of Stokes-filtered local systems (of exponential type) on the circle at infinity $S^1_\infty$ of $\Afu_\tau$.\qed
\end{proposition}

Of course, this result gives back the abelianity result of Proposition \ref{prop:catStokesfiltabelian} (proved directly for any field of coefficients).

\subsection{Stokes data}\label{subsec:Stokesdata}
These are linear data which provide a description of \addedo{a}~Stokes-filtered local system.\addedo{ Let~$C$ be a non-empty finite subset of $\CC$. We say that $\theta_o\in\RR/2\pi\ZZ$ is \emph{generic} with respect to $C$ if \addedn{it} is a Stokes direction (\cf Remark \ref{rem:Stokesdir}\eqref{rem:Stokesdir2}) for no pair $c\neq c'\in C$. Once~$\theta_o$ generic with respect to~$C$ is chosen, there is a unique numbering of the set~$C$ in such a way that $c_1\lethetao c_2\lethetao\cdots\lethetao c_r$. We will set $\theta'_o=\theta_o+\pi$. Note that the order is exactly reversed at~$\theta'_o$, so that $-C$ is numbered as $\{-c_1,\dots,-c_n\}$ by $\theta'_o$.}

\begin{definition}\label{def:catStokesdata}
Let~$C$ be a non-empty finite subset of $\CC$\addedo{ and let $\theta_o\in\RR/2\pi\ZZ$ be generic with respect to $C$}. The category of Stokes data with exponential factors in~$C$ \addedo{totally ordered by~$\theta_o$ (we also say \emph{of type $(C,\theta_o)$})} has objects consisting of two families of $\kk$-vector spaces $(G_{c,1},G_{c,2})_{c\in C}$ and a diagram of morphisms
\begin{equation}\tag*{(\protect\ref{def:catStokesdata})$(*)$}\label{eq:catStokesdata}
\begin{array}{c}
\xymatrix@C=1.5cm{
\tbigoplus_{c\in C}G_{c,1}\ar@/^1pc/[r]^-{S}\ar@/_1pc/[r]_-{S'}&\tbigoplus_{c\in C}G_{c,2}
}
\end{array}
\end{equation}
such that, for \replacedo{the}{some} numbering $C=\{c_1,\dots,c_n\}$\addedo{ defined by~$\theta_o$},
\begin{enumerate}
\item
$S=(S_{ij})_{i,j=1,\dots, n}$ is block-upper triangular, \ie $S_{ij}:G_{c_i,1}\to G_{c_j,2}$ is zero unless $i\leq j$, and $S_{ii}$ is invertible (so $\dim G_{c_i,1}=\dim G_{c_i,2}$, and $S$ itself is invertible),
\item
$S'=(S'_{ij})_{i,j=1,\dots, n}$ is block-lower triangular, \ie $S'_{ij}:G_{c_i,1}\to G_{c_j,2}$ is zero unless $i\geq j$, and $S'_{ii}$ is invertible (so $S'$ itself is invertible).
\end{enumerate}

A morphism of Stokes data \addedo{of type $(C,\theta_o)$} consists of morphisms of $\kk$-vector spaces $\lambda_{c,\ell}:G_{c,\ell}\to G'_{c,\ell}$, $c\in C$, $\ell=1,2$ which are compatible with the \addedo{corresponding} diagrams \ref{eq:catStokesdata}.
\end{definition}

Fixing bases in the spaces $G_{c,\ell}$, $c\in C$, $\ell=1,2$, allows one to present Stokes data by matrices $(\Sigma,\Sigma')$ where $\Sigma=(\Sigma_{ij})_{i,j=1,\dots, n}$ (\resp $\Sigma'=(\Sigma'_{ij})_{i,j=1,\dots, n}$) is block-upper (\resp -lower) triangular and each $\Sigma_{ii}$ (\resp $\Sigma'_{ii}$) is invertible.

The category of Stokes data \replacedo{of type $(C,\theta_o)$}{with exponential factors contained in~$C$} is clearly abelian. We will now define a functor \addedo{(depending on~$\theta_o$)} from the category of Stokes-filtered local systems with exponential factors contained in~$C$ to the category of Stokes data \replacedo{of type $(C,\theta_o)$}{with exponential factors contained in~$C$}, and we will show that it is an equivalence. In the next section, we will show that it is compatible with natural operations on these objects (involution $\iota$, duality, sesquilinear duality).

\deletedo{The functor depends on the choice of $\theta_o\in\RR/2\pi\ZZ$, which is not a Stokes direction (\cf Remark \ref{rem:Stokesdir}\eqref{rem:Stokesdir2}) for any pair $c\neq c'\in C$. Once~$\theta_o$ is chosen, we choose a numbering of the $c\in C$ in such a way that $c_1\lethetao c_2\lethetao\cdots\lethetao c_r$. We will set $\theta'_o=\theta_o+\pi$. Note that the order is exactly reversed at~$\theta'_o$.}

Let us also fix two opposite intervals~$I_1$ and $I_2$ of length $\pi+\addedo{2}\epsilon$ on $\addedo{\RR/2\pi\ZZ}$ so that their intersection $I_1\cap I_2$ consists of two intervals $(\theta_o-\epsilon,\theta_o+\epsilon)$ and $(\theta'_o-\epsilon,\theta'_o+\epsilon)$, and contain\addedo{s} no Stokes direction of pairs $c\neq c'\in C$.

To a local system~$\cL$ on $S^1$ we attach the following ``monodromy data'' (they are quite redundant):
\begin{enumerate}
\item
vector spaces $L_1=\Gamma(I_1,\cL)$ and $L_2=\Gamma(I_2,\cL)$,
\item
vector spaces $L_{\theta_o}=\cL_{\theta_o}$ and $L_{\theta'_o}=\cL_{\theta'_o}$,
\item
a diagram of isomorphisms, given by the natural restriction morphisms,
\[
\xymatrix{
&\ar[ld]_-{a'_1}L_1\ar[rd]^-{a_1}&\\
L_{\theta'_o}&&L_{\theta_o}\\
&\ar[lu]^-{a'_2}L_2\ar[ru]_{a_2}&
}
\]
\end{enumerate}

This reduces to two possible descriptions:
\begin{enumeratea}
\item\label{enum:Stokesdataa}
$(L_1,L_2,S_{\theta_o},S_{\theta'_o})$, with isomorphisms $S_{\theta_o},S_{\theta'_o}:L_1\isom L_2$ and monodromy $T_1:L_1\isom L_1$, where
\[
S_{\theta_o}=a_2^{-1}a_1,\quad S_{\theta'_o}=a_2^{\prime-1}a'_1,\quad T_1=S_{\theta_o}^{-1}S_{\theta'_o}.
\]
\item\label{enum:Stokesdatab}
$(L_{\theta_o}, L_{\theta'_o},S_1,S_2)$ with isomorphisms $S_1,S_2:L_{\theta_o}\isom L_{\theta'_o}$ and monodromy $T_{\theta_o}:L_{\theta_o}\isom L_{\theta_o}$, where
\[
S_1=a'_1a_1^{-1},\quad S_2=a'_2a_2^{-1},\quad T_{\theta_o}=S_2^{-1}S_1.
\]
\end{enumeratea}

Assume now that $(\cL,\cL_\bbullet)$ is a Stokes-filtered local system with associated graded local system $\cG=\gr\cL=\bigoplus_{i=1}^r\cG_{c_i}$. The filtration $\cL_{\leq c,\theta_o}$ induces a filtration on $L_{\theta_o}$ and, through $a_1$, a filtration $L_{1,\leqthetao \cbbullet}$ of $L_1$. We have a similar filtration attached to~$\theta'_o$.

We have splittings (\cf Proposition \ref{prop:strict}\eqref{prop:strict1}\deletedo{ and Remark \ref{rem:uniquesplitting}}):
\[
(\cL,\cL_\bbullet)_{|I_1}\simeq(\cG,\cG_\bbullet)_{|I_1},\quad (\cL,\cL_\bbullet)_{|I_2}\simeq(\cG,\cG_\bbullet)_{|I_2},
\]
giving isomorphisms
\begin{equation}\label{eq:splittings}
L_1\simeq\tbigoplus_{i=1}^rG_{c_i,1},\quad L_2\simeq\tbigoplus_{i=1}^rG_{c_i,2}
\end{equation}
compatible with Stokes filtrations \addedo{(in other words, both filtrations $L_{1,\leqthetao \cbbullet}$ and $L_{1,\leqthetapo \cbbullet}$ are opposite in $L_1$, \cf Remark \ref{rem:splitting}\eqref{rem:splitting1}, giving rise to a unique common splitting, and similarly for $L_2$)}, and such that $S_{\theta_o}$ (\resp $S_{\theta'_o}$) is compatible with the \replacedo{filtration}{order} at~$\theta_o$ (\resp~$\theta'_o$) and the graded morphisms are isomorphisms. Taking into account the assumption on the ordering of the~$c_j$, this is equivalent to saying that $S_{\theta_o}$ is block-upper triangular, $S_{\theta'_o}$ is block-lower triangular, and each diagonal block $\gr_{c_i}S_{\theta_o}$, $\gr_{c_i}S_{\theta'_o}$ is an isomorphism. In such a way, we have defined the desired functor (to check the compatibility with morphisms, use Proposition \ref{prop:strict}\eqref{prop:strict2}). The Stokes data attached to $(\cL,\cL_\bbullet)$ are given by the diagram:
\begin{equation}\label{eq:Stokesdata}
\begin{array}{c}
\xymatrix@C=1.5cm{
\tbigoplus_{i=1}^rG_{c_i,1}\ar@/^1pc/[r]^-{S_{\theta_o}}\ar@/_1pc/[r]_-{S_{\theta'_o}}&\tbigoplus_{i=1}^rG_{c_i,2}
}
\end{array}
\end{equation}
Note also that the monodromy $T_{c_i}$ on $\gr_{c_i}\cL$ is given by $T_{c_i,1}=(\gr_{c_i}S_{\theta_o})^{-1}\gr_{c_i}S_{\theta'_o}$ (this is of course not obtained from the bloc\addedo{k}s of $T_1=S_{\theta_o}^{-1}S_{\theta'_o}$ in general).

As a consequence of the previous discussion we can state the following classical result (the bijection at the level of $\Hom$ follows from Proposition \ref{prop:strict}\eqref{prop:strict2}):

\begin{proposition}\label{prop:datafiltered}
The previous functor is an equivalence between the category of Stokes-filtered local systems with exponential factors contained in~$C$ and the category of Stokes data \replacedo{of type $(C,\theta_o)$}{with exponential factors contained in~$C$}.\qed
\end{proposition}

\begin{definition}[\addedm{minimality property}]\label{def:critminext}
\replacedm{We say that the Stokes data \ref{eq:catStokesdata} satisfy the \emph{minimality property} if the vector space $K_c\defin\{v\in G_{c,1}\mid S(v)=S'(v)\in G_{c,2}\}$ is equal to zero for any $c\in C$.}{We say that the Stokes data \ref{eq:catStokesdata} do not satisfy the minimality property if there exist $c\in C$ and $v\in G_{c,1}\addedo{\moins\{0\}}$\deletedo{ in} such that $S(v)=S'(v)\in G_{c,2}$.}
\end{definition}
\addedo{\begin{remarque}\label{rem:minassumption}
Notice that, if $S-S'$ is invertible, the minimality property is automatically satisfied. Notice also that \replacedm{$K_c$ is the subspace of $G_{c,1}$ consisting of eigenvectors of~$T_1$ (and thus of $T_{c,1}$) with eigenvalue $1$}{$v\in G_{c,1}\moins\{0\}$ breaks the minimality property if and only if $v$ is an eigenvector of $T_1$ (and thus of $T_{c,1}$) with eigenvalue $1$}.
\end{remarque}}

\begin{lemme}\label{lem:equivminL}
Under the equivalence of Proposition \ref{prop:datafiltered}, the Stokes data attached to $(\cL,\cL_\bbullet)$ satisfy the minimality property of Definition \ref{def:critminext} if and only if $(\cL,\cL_\bbullet)$ has no subobject $(\cL',\cL'_\bbullet)$ (in the category of Stokes-filtered local systems) such that \hbox{$\cL'=\kk_{S^1}$}.
\end{lemme}

\begin{proof}
A subobject $(\cL',\cL'_\bbullet)$ of $(\cL,\cL_\bbullet)$ corresponds to a subdiagram of \eqref{eq:Stokesdata} compatible with the splittings \eqref{eq:splittings}. The condition in Definition \ref{def:critminext} amounts to the existence of a subdiagram
\begin{equation}
\tag*{\protect(\ref{lem:equivminL})($*$)}\label{eq:equivminL}
\xymatrix@C=1.5cm{
\kk\cdot v\ar@/^1pc/[r]^-{S_{\theta_o}}\ar@/_1pc/[r]_-{S_{\theta'_o}}&\hspace*{4mm}\kk\cdot(S_{\theta_o} v)
}
\end{equation}
of \eqref{eq:Stokesdata} with $v\in G_{c_{j_o},1}$ and $S_{\theta_o} v=S_{\theta'_o} v\in G_{c_{j_o},2}$, therefore compatible with the splittings \eqref{eq:splittings}, and $v\in L_1$ satisfies $T_1v=v$. Therefore it corresponds to $(\cL',\cL'_\bbullet)\subset(\cL,\cL_\bbullet)$ with $\cL'=\kk_{S^1}$.
\end{proof}
\addedo{\begin{remarque}
Definition \ref{def:critminext} and Lemma \ref{lem:equivminL} fit to Lemma \ref{lem:critminext} and to the property that~$M$ is a minimal extension via Propositions \ref{prop:equivconnStokesloc} and \ref{prop:datafiltered}.
\end{remarque}}

\section{Natural operations on Stokes filtrations and Stokes data}\label{sec:operations}
\subsection{Involution}\label{subsec:involution}
Let~$\iota$ be the involution $z\mto-z$, which is induced on $\addedo{\RR/2\pi\ZZ}$ by $\theta\mto\theta'\defin\theta+\pi$. Given a Stokes-filtered local system $(\cL,\cL_\bbullet)$, we define $\iota^{-1}(\cL,\cL_\bbullet)$ in the following way:
\begin{itemize}
\item
the corresponding local system is $\iota^{-1}\cL$, so that $(\iota^{-1}\cL)_\theta=\cL_{\theta'}$,
\item
the filtration $(\iota^{-1}\cL)_\bbullet$ is defined by $(\iota^{-1}\cL)_{\leq c}=\iota^{-1}(\cL_{\leq-c})$, hence $(\iota^{-1}\cL)_{\leq c,\theta}=\cL_{\leq-c,\theta'}$.
\end{itemize}
Note that the filtration defined above is increasing, that is, $c'\leqtheta c\iff -c'\leqthetap -c$. The monodromy data $\iota^{-1}(L_1,L_2,S_{\theta_o},S_{\theta'_o})$ of $\iota^{-1}\cL$ are given by $(L_2,L_1,S_{\theta'_o}^{-1},S_{\theta_o}^{-1})$. The Stokes data of $\iota^{-1}(\cL,\cL_\bbullet)$ are given by
\begin{equation}\tag*{$\iota^{-1}(\protect\ref{eq:Stokesdata})$}\label{eq:iotaStokesdata}
\begin{array}{c}
\xymatrix@C=1.5cm{
\tbigoplus_{i=1}^rG_{c_i,2}\ar@/^1pc/[r]^-{S_{\theta'_o}^{-1}} \ar@/_1pc/[r]_-{S_{\theta_o}^{-1}}&\tbigoplus_{i=1}^rG_{c_i,1}.
}
\end{array}
\end{equation}
\deletedo{and one should write $G_{c_i,2}=(\iota^{-1}G)_{-c_i,1}$, and vice-versa, to understand that $S_{\theta'_o}^{-1}$ is upper triangular.}

In other words, \replacedo{\ref{eq:iotaStokesdata} defines a functor $\iota$ from the category of Stokes data of type $(C,\theta_o)$ to that of type $(-C,\theta'_o)$, and the corresponding equivalences of Proposition \ref{prop:datafiltered} are compatible with $\iota$ on both categories}{defining the action of $\iota$ as above on the category of Stokes data with exponential factors contained in~$C$, the equivalence of Proposition \ref{prop:datafiltered} is compatible with $\iota$}.

\addedn{Let us note that, although the local systems $\cL$ and $\iota^{-1}\cL$ are isomorphic (since $\iota$ is homotopic to the identity), the Stokes-filtered local systems $(\cL,\cL_\bbullet)$ and $\iota^{-1}(\cL,\cL_\bbullet)$ are in general not isomorphic. For example, they are isomorphic if both $S$ and $S'$ are block-diagonal, an isomorphism of the corresponding Stokes data being given by the pair of morphisms $(S^{\prime-1}SS^{\prime-1},SS^{\prime-1}S)$.}

\subsection{Duality}\label{subsec:duality}
Let $(\cL,\cL_\bbullet)$ be a Stokes-filtered local system. The dual local system $\cL^\vee$ comes equipped with a \deletedo{Stokes} filtration $(\cL^\vee)_\bbullet$ defined by
\[
(\cL^\vee)_{\leq c}=(\cL_{<-c})^\perp,
\]
where the orthogonality is relative to duality\addedo{ that is, $(\cL_{<-c})^\perp$ consists of local morphisms $\cL\to\kk_{S^1}$ sending $\cL_{<-c}$ to~$0$. Using, in a neighbourhood of $e^{i\theta}\in S^1$, a local splitting of $\cL$ as $\bigoplus_{c_i\in C}\gr_{c_i}\cL$ compatible with the Stokes filtration, we get a corresponding local splitting $\cL^\vee\simeq\bigoplus_{c_i\in C}(\gr_{c_i}\cL)^\vee$, and a germ at $e^{i\theta}$ of a morphism $\varphi$ has components $\varphi_i$. Then $\varphi\in(\cL_{<-c})^\perp_\theta$ if and only if its components $\varphi_i$ vanish whenever $\beta_{c_i<-c}\gr_{c_i}\cL\neq0$ somewhere near $\theta$. So the only possible nonzero components~$\varphi_i$ of $\varphi$ occur when $-c_i\leqtheta c$. If we set $\gr_{-c_i}\cL^\vee\defin(\gr_{c_i}\cL)^\vee$, this shows that $(\cL^\vee)_{\leq c} $ locally splits near $e^{i\theta}$ as $\bigoplus_i\beta_{-c_i\leq c}\gr_{-c_i}\cL^\vee$, defining thus a Stokes filtration satisfying $\gr_c(\cL^\vee)=(\gr_{-c}\cL)^\vee$ for any $c\in \CC$}. The monodromy data $(L_1,L_2,S_{\theta_o},S_{\theta'_o})^\vee$ are given by $(L_1^\vee,L_2^\vee,\tS_{\theta_o}^{-1},\tS_{\theta'_o}^{-1})$, where $\tS$ denotes the adjoint by duality of $S$. The Stokes data are given by
\begin{equation}\tag*{$(\protect\ref{eq:Stokesdata})^\vee$}\label{eq:tStokesdata}
\begin{array}{c}
\xymatrix@C=1.5cm{
\tbigoplus_{i=1}^r(G_{c_i,1})^\vee\ar@/^1pc/[r]^-{\tS_{\theta_o}^{-1}} \ar@/_1pc/[r]_-{\tS_{\theta'_o}^{-1}}&\tbigoplus_{i=1}^r(G_{c_i,2})^\vee.
}
\end{array}
\end{equation}
\deletedo{where $\tS_{\theta_o}^{-1}$ is upper triangular if one sets $(G_{c_i,k})^\vee=(G^\vee)_{-c_i,k}$, $k=1,2$.} Let us define \addedo{the} duality \addedo{functor} \replacedo{from}{in} the category of Stokes data \addedo{of type $(C,\theta_o)$ to that of type $(-C,\theta_o)$} by the previous formula\addedo{ (because we use the reverse numbering of $C$ to get that $\tS_{\theta_o}^{-1}$ is upper triangular, that is, the numbering of $-C$ induced by~$\theta_o$)}. Then the equivalence of Proposition \ref{prop:datafiltered} is compatible with duality.

Let us now compare with Poincaré-Verdier duality of sheaves on $S^1$. For a sheaf~$\cF$ on $S^1$, we denote by $\DD\cF=\bR\cHom_{\kk}(\cF,\kk_{S^1}[1])$ its Poincaré-Verdier dual and by $\DD'\cF=\bR\cHom_{\kk}(\cF,\kk_{S^1})$ the shifted complex. We clearly have $\DD'\cL=\cL^\vee$.

\begin{lemme}\label{lem:dualS1}
For each $c\in\CC$, the complexes $\DD'(\cL_{\leq c})$ and $\DD'(\cL/\cL_{\leq c})$ are sheaves and $\DD'(\cL/\cL_{\leq c})=(\cL_{\leq c})^\perp=(\cL^\vee)_{<-c}$.
\end{lemme}

The first statement means that $\cH^k\DD'(\cL_{\leq c})=0$ if $k\neq0$, and thus $\DD'(\cL_{\leq c})$ is quasi-isomorphic to $\cH^0\DD'(\cL_{\leq c})=\cHom_{\kk}(\addedo{\cL_{\leq c}},\kk_{S^1})$ and similarly for $\cL/\cL_{\leq c}$.

\begin{proof}
The first assertion is local on $S^1$, so we can assume that~$\cL$ is split with respect to the Stokes filtration. Near $\theta_o\in\RR/2\pi\ZZ$, we therefore only need to consider \replacedo{two}{three} cases:
\begin{enumerate}
\item
$\cL_{\leq c}$ is a local system on $(\theta_o-\epsilon,\theta_o+\epsilon)$,

\deletedo{(2) $\cL_{\leq c}=0$ on $(\theta_o-\epsilon,\theta_o+\epsilon)$,}
\item
$\cL_{\leq c}=j_!\cL$, where~$\cL$ is a local system on $(\theta_o-\epsilon,\theta_o)$ and $j: (\theta_o-\epsilon,\theta_o)\hto(\theta_o-\epsilon,\theta_o+\epsilon)$ is the open inclusion.
\end{enumerate}
The first \replacedo{case is}{two cases are} clear. For the \replacedo{second}{third} one, note that $\DD'j_!\cL=j_*\cL^\vee$.

The argument for $\DD'(\cL/\cL_{\leq c})$ is similar (but goes in the opposite direction in the \replacedo{second}{third} case). We conclude that we have an exact sequence of sheaves
\[
0\to\DD'(\cL/\cL_{\leq c})\to\DD'\cL\to\DD'(\cL_{\leq c})\to0,
\]
hence the last assertion.
\end{proof}

\subsection{$\iota$-Sesquilinear forms}\label{subsec:sesq}
We assume here that $\kk=\CC$ (or that $\kk$ has an involution, that we denote by $\overline{\phantom{X}}$). Let $\varh:\cL\otimes\iota^{-1}\ov\cL\to\kk$ be linear, where $\ov\cL$ denotes the conjugate of~$\cL$ with respect to the involution (in what follows, one can assume that the involution is the identity and get similar results for $\iota$-bilinear forms). We call~$\varh$ a $\iota$-sesquilinear form on~$\cL$. Using the previous monodromy data, giving~$\varh$ amounts to giving two sesquilinear forms
\[
\varh_{1\ov2}:L_1\otimes\ov L_2\to\kk,\qquad
\varh_{2\ov1}:L_2\otimes\ov L_1\to\kk
\]
such that, considering them as morphisms $\ov L_2\to L_1^\vee$ and $\ov L_1\to L_2^\vee$, the following diagrams commute (defining $\varh_{\theta_o},\varh_{\theta'_o}$)
\begin{equation}\label{eq:hhh}
\begin{array}{c}
\xymatrix{
\ov L_2\ar[r]^-{\varh_{1\ov2}}\ar[d]_{\ov a'_2}&L_1^\vee\ar[d]^{{}^t\!a_1^{-1}}\\
\ov L_{\theta'_o}\ar[r]^-{\varh_{\theta_o}}&L_{\theta_o}^\vee\\
\ov L_1\ar[r]^-{\varh_{2\ov1}}\ar[u]^{\ov a'_1}&L_2^\vee\ar[u]_{{}^t\!a_2^{-1}}
}
\end{array}
\qquad\qquad
\begin{array}{c}
\xymatrix{
\ov L_2\ar[r]^-{\varh_{1\ov2}}\ar[d]_{\ov a_2}&L_1^\vee\ar[d]^{{}^t\!a_1^{\prime-1}}\\
\ov L_{\theta_o}\ar[r]^-{\varh_{\theta'_o}}&L_{\theta'_o}^\vee\\
\ov L_1\ar[r]^-{\varh_{2\ov1}}\ar[u]^{\ov a_1}&L_2^\vee\ar[u]_{{}^t\!a_2^{\prime-1}}
}
\end{array}
\end{equation}
that is,
\begin{equation}\label{eq:h12h21}
\varh_{2\ov1}(\cbbullet,\ov\cbbullet)=\varh_{1\ov2}(S_{\theta_o}^{-1}\cbbullet,\ov{S_{\theta'_o}\cbbullet})=\varh_{1\ov2}(S_{\theta'_o}^{-1}\cbbullet,\ov{S_{\theta_o}\cbbullet}).
\end{equation}
In particular, $\varh_{2\ov1}$ determines $\varh_{1\ov2}$. We say that~$\varh$ is nondegenerate if it induces an isomorphism $\iota^{-1}\ov\cL\isom\cL^\vee$, that is, if $\varh_{2\ov1}$ (hence $\varh_{1\ov2}$) is nondegenerate. We say that~$\varh$ is $\iota$-skew-Hermitian if $\iota^{-1}\ov \varh=-\varh$ (with an obvious meaning), that is, if
\begin{equation}\label{eq:h21ovh12}
\varh_{2\ov1}(x_2,\ov x_1)=-\ov{\varh_{1\ov2}(x_1,\ov x_2)}.
\end{equation}

\begin{remarque}[Various forms of $\varh$]\label{rem:11th}
\setcounter{equation}{0}
\let\oldtheequation\theequation
\def\theequation{\alph{equation}}
It will be useful to read $\varh$ in the spaces $L_{\theta_o},L'_{\theta_o}$ or only in $L_1$. We will make explicit the formulas between the various forms. We denote by $x_1,y_1$ general elements of $L_1$, $x,y$ of $L_{\theta_o}$ and $x',y'$ of $L'_{\theta_o}$. Firstly, \eqref{eq:hhh} gives
\begin{equation}\refstepcounter{equation}\tag*{(\protect\ref{rem:11th})(\theequation)}\label{eq:varhthetao}
\begin{split}
\varh_{\theta_o}(x,\ov{y'})&=\varh_{1\ov2}(a_1^{-1}x,\ov{a_2^{\prime-1}y'})\addedo{{}=\varh_{2\ov1}(a_2^{-1}x,\ov{a_1^{\prime-1}y'}),}\\
\varh_{\theta'_o}(x',\ov y)&=\varh_{1\ov2}(a_1^{\prime-1}x',\ov{a_2^{-1}y})\addedo{{}=\varh_{2\ov1}(a_2^{\prime-1}x',\ov{a_1^{-1}y})}.
\end{split}
\end{equation}
Let us define $\varh_{1\ov1}^{\theta_o},\varh_{1\ov1}^{\theta'_o}:L_1\otimes\ov L_1\to\kk$ by
\begin{align}\refstepcounter{equation}\tag*{(\protect\ref{rem:11th})(\theequation)}\label{eq:varh1ov1}
\varh_{1\ov1}^{\theta_o}(x_1,\ov y_1)&=\varh_{1\ov2}(x_1,\ov{S_{\theta_o}y_1}),&
\varh_{1\ov1}^{\theta'_o}(x_1,\ov y_1)&=\varh_{1\ov2}(x_1,\ov{S_{\theta'_o}y_1}).
\end{align}
Then
\begin{equation}\refstepcounter{equation}\tag*{(\protect\ref{rem:11th})(\theequation)}\label{eq:varh1ov1prime}
\varh_{1\ov1}^{\theta'_o}(x_1,\ov y_1)=\varh_{1\ov1}^{\theta_o}(x_1,\ov{T_1y_1}),
\end{equation}
and \eqref{eq:h12h21} is equivalent to
\begin{equation}\refstepcounter{equation}\tag*{(\protect\ref{rem:11th})(\theequation)}\label{eq:varh1ov1T1}
\varh_{1\ov1}^{\theta_o}(T_1x_1,\ov{T_1y_1})=\varh_{1\ov1}^{\theta_o}(x_1,\ov y_1)\quad\addedo{\text{and to}}\quad
\varh_{1\ov1}^{\theta'_o}(T_1x_1,\ov{T_1y_1})=\varh_{1\ov1}^{\theta'_o}(x_1,\ov y_1).
\end{equation}
We also get
\begin{equation}\refstepcounter{equation}\tag*{(\protect\ref{rem:11th})(\theequation)}\label{eq:varh1ov1thetao}
\addedo{\varh_{1\ov1}^{\theta'_o}(x_1,\ov y_1)}=\varh_{\theta_o}(a_1x_1,\ov{a'_1y_1})\quad\addedo{\text{and}}\quad
\addedo{\varh_{1\ov1}^{\theta_o}(x_1,\ov y_1)}=\varh_{\theta'_o}(a'_1x_1,\ov{a_1y_1}).
\end{equation}
Moreover, $\varh$ is $\iota$-skew-Hermitian iff
\begin{equation}\refstepcounter{equation}\tag*{(\protect\ref{rem:11th})(\theequation)}\label{eq:11th}
\varh_{1\ov1}^{\theta_o}(y_1,\ov x_1)=-\ov{\varh_{1\ov1}^{\theta'_o}(x_1,\ov y_1)}.
\end{equation}
\let\theequation\oldtheequation
\setcounter{equation}{5}
\end{remarque}

\begin{remarque}[The form induced on $\im\can_1$]\label{rem:imcan}
\setcounter{equation}{0}
\let\oldtheequation\theequation
\def\theequation{\alph{equation}}
Let us set $\can_1:=\id-T_1:L_1\to L_1$. We have the following relations
\begin{equation}\refstepcounter{equation}\tag*{(\protect\ref{rem:imcan})(\theequation)}\label{eq:h1ov1can1}
\begin{split}
\varh_{1\ov1}^{\theta_o}(x_1,\ov{\can_1y_1})&=\varh_{1\ov1}^{\theta_o}(x_1,\ov y_1)-\varh_{1\ov1}^{\theta_o}(x_1,\ov{T_1y_1})\\
&=\varh_{1\ov1}^{\theta_o}(x_1,\ov y_1)-\varh_{1\ov1}^{\theta'_o}(x_1,\ov y_1)\quad\text{after \ref{eq:varh1ov1prime}},
\end{split}
\end{equation}
and similarly
\begin{equation}\refstepcounter{equation}\tag*{(\protect\ref{rem:imcan})(\theequation)}\label{eq:h1ov1can1prime}
\varh_{1\ov1}^{\theta'_o}(\can_1x_1,\ov y_1)=\varh_{1\ov1}^{\theta'_o}(x_1,\ov y_1)-\varh_{1\ov1}^{\theta_o}(x_1,\ov y_1)=-\varh_{1\ov1}^{\theta_o}(x_1,\ov{\can_1y_1}).
\end{equation}
Let us set $F_1=\im\can_1$. Then $\varh_{1\ov1}^{\theta_o}$ defines a sesquilinear pairing $\hh_{1\ov1}^{\theta_o}$ on $F_1$ by setting, for $u_1,v_1\in F_1$ and $u_1=\can_1x_1$, $v_1=\can_1y_1$ for some $x_1,y_1\in L_1$:
\[
\hh_{1\ov1}^{\theta_o}(u_1,\ov v_1)\defin\varh_{1\ov1}^{\theta_o}(x_1,\ov v_1).
\]
This is independent of the choice of $x_1$: if $\can x_1=0$, we deduce from \ref{eq:h1ov1can1prime}
\[
\varh_{1\ov1}^{\theta_o}(x_1,\ov v_1)=\varh_{1\ov1}^{\theta_o}(x_1,\ov{\can_1y_1})=-\varh_{1\ov1}^{\theta'_o}(\can_1x_1,\ov y_1)=0.
\]
We also set $\hh_{1\ov1}^{\theta'_o}(u_1,\ov v_1)\defin\varh_{1\ov1}^{\theta'_o}(u_1,\addedo{\ov y_1})$. Then $\hh_{1\ov1}^{\theta'_o}(u_1,\ov v_1)=-\hh_{1\ov1}^{\theta_o}(u_1,\ov v_1)$.

If $\varh$ is nondegenerate, then so is $\hh_{1\ov1}^{\theta_o}$ on $F_1$: assume that $\hh_{1\ov1}^{\theta_o}(u_1,\ov v_1)=0$ for all $v_1\in F_1$. Then $\varh_{1\ov1}^{\theta_o}(x_1,\ov{\can_1y_1})=0$ for all $y_1\in L_1$, and as above this implies that $u_1=\can_1x_1=0$ since
$\varh_{1\ov1}^{\theta_o}$ is nondegenerate on $L_1$.

Lastly, if $\varh$ is $\iota$-skew-Hermitian, then $\hh_{1\ov1}^{\theta_o}$ is Hermitian on $F_1$: we have
\begin{align*}
\hh_{1\ov1}^{\theta_o}(v_1,\ov u_1)&=\varh_{1\ov1}^{\theta_o}(y_1,\ov{\can_1x_1})=\varh_{1\ov1}^{\theta_o}(y_1,\ov x_1)-\varh_{1\ov1}^{\theta'_o}(y_1,\ov x_1)\\
&=-\ov{\varh_{1\ov1}^{\theta'_o}(x_1,\ov y_1)}+\ov{\varh_{1\ov1}^{\theta_o}(x_1,\ov y_1)}\quad\text{after \ref{eq:11th}}\\
&=\ov{\hh_{1\ov1}^{\theta_o}(u_1,\ov v_1)}.
\end{align*}
\let\theequation\oldtheequation
\setcounter{equation}{6}
\end{remarque}

\subsection{$\iota$-Sesquilinear forms on Stokes-filtered local systems and Stokes data}\label{subsec:sesqstokes}
If $(\cL,\cL_\bbullet)$ is a Stokes-filtered local system, we say that~$\varh$ is \emph{compatible with Stokes filtrations} if the induced morphism $\iota^{-1}\ov\cL\to\cL^\vee$ is so. By Proposition \ref{prop:strict}\eqref{prop:strict2}, $\varh_{1\ov2}$ and $\varh_{2\ov1}$ are block-diagonal. Similarly, given Stokes data \replacedo{$((G_{c,1},G_{c,2})_{c\in C},S,S')$ of type $(C,\theta_o)$}{$(C,G,S)$}, a $\iota$-sesquilinear form on it consists of sesquilinear pairings $\varh_{1\ov2}^{(i)}:G_{c_i,1}\otimes\ov G_{c_i,2}\to\kk$ (and similarly for $\varh_{2\ov1}$) which are compatible with the diagram \ref{eq:catStokesdata} in a natural way. In other words, the equivalence of Proposition \ref{prop:datafiltered} is compatible with $\iota$-sesquilinear forms.

Let us fix~$c\in\CC$ such that
\begin{enumeratea}
\item\label{assumpt:a}
$\gr_c\cL=0$ (that is, $\cL_{<c}=\cL_{\leq c}$).
\end{enumeratea}
According to \eqref{assumpt:a}, the morphism $\varh:\iota^{-1}\ov\cL\to\cL^\vee$ induces
\[
\varh_c:\iota^{-1}(\ov\cL_{\leq c})=(\iota^{-1}\ov\cL)_{\leq -c}\to(\cL^\vee)_{\leq -c}=(\cL_{\leq c})^\perp=(\cL/\cL_{\leq c})^\vee,
\]
that we consider as a pairing
\begin{equation}\label{eq:hc}
\varh_c:(\cL/\cL_{\leq c})\otimes\iota^{-1}(\ov\cL_{\leq c})\to\kk_{S^1}
\end{equation}
Moreover, if~$\varh$ is nondegenerate, then~$\varh_c$ is nondegenerate in the sense that~$\varh_c$ induces an isomorphism
\begin{equation}\label{eq:hcnondeg}
\iota^{-1}(\ov\cL_{\leq c})\isom\DD'(\cL/\cL_{\leq c})=\cHom_{\kk}(\cL/\cL_{\leq c},\kk_{S^1})=(\cL/\cL_{\leq c})^\vee.
\end{equation}

\begin{remarque}[The form induced on $K_c$]\label{rem:Kc}
\addedm{For $c\in C$, let $K_c\subset G_{c,1}$ be the vector space introduced in Definition \ref{def:critminext}. Together with $S$ or $S'$, the sesquilinear form $\varh_{1\ov2}$ produces a sesquilinear form $\varh_{K_c}:K_c\otimes\ov K_c\to\kk$ by the formula $\varh_{K_c}(x_1,\ov y_1)=\varh_{1\ov2}(x_1,\ov{S(y_1)})$. Since $S=S'$ on $K_c$, the form $\varh_{K_c}$ is skew-Hermitian, according to \eqref{eq:h12h21} and \eqref{eq:h21ovh12}. We claim that, for $c_o\in C$, \emph{the form $\varh_{K_{c_o}}$ is nondegenerate if and only if the Stokes data $(K_{c_o},S(K_{c_o}),S,S)$ enriched with the induced $\varh_{1\ov2}$ are a direct summand of the Stokes data $((G_{c,1})_{c\in C},(G_{c,2})_{c\in C},S,S')$ enriched with $\varh_{1\ov2}$}.}

\addedm{Indeed, since $\varh_{1\ov2}$ is block-diagonal, $\varh_{K_{c_o}}$ is nondegenerate if and only if $K_{c_o}\cap\nobreak S(K_{c_o})^\perp=\{0\}$, where the orthogonal is taken with respect to $\varh_{1\ov2}$. Notice that $S(K_{c_o})^\perp=(S(K_{c_o})^\perp\cap G_{c_o,1})\oplus\bigoplus_{c\neq c_o\in C}G_{c,1}$. A similar statement holds for $S(K_{c_o})$ and $K_{c_o}^\perp$ in $\bigoplus_{c\in C}G_{c,2}$. Then, if $\varh_{K_{c_o}}$ is nondegenerate, we have $\bigoplus_{c\in C}G_{c,1}=K_{c_o}\oplus S(K_{c_o})^\perp$, $\bigoplus_{c\in C}G_{c,2}=S(K_{c_o})\oplus K_{c_o}^\perp$, and it remains to check that $S$ and $S'$ send $S(K_{c_o})^\perp$ to $K_{c_o}^\perp$, which follows from \eqref{eq:h12h21} and \eqref{eq:h21ovh12}. The converse is proved similarly.}

\addedm{We also notice that, if the previous splitting property is satisfied for each $c\in C$, then
\begin{multline*}
((G_{c,1})_{c\in C},(G_{c,2})_{c\in C},S,S',\varh_{1\ov2})\\
=\Big(\tbigoplus_{c\in C}(K_c,S(K_c),S,S,\varh_{1\ov2})\Big)\oplus((G'_{c,1})_{c\in C},(G'_{c,2})_{c\in C},S,S',\varh'_{1\ov2}),
\end{multline*}
where the last term satisfies the minimality property of Definition \ref{def:critminext}. Indeed, we then have $K'_c=0$ for any $c\in C$.
}
\end{remarque}

\subsection{Description at $\theta_o,\theta'_o$}\label{subsec:htheta}
Let us fix bases of $G_{c_i,k}$, $k=1,2$, $i=1,\dots,r$, in such a way that the matrix of $\varh_{1\ov2}^{(i)}$ is the identity. If we denote by $\Sigma_{\theta_o},\Sigma_{\theta'_o}$ the matrices of $S_{\theta_o},S_{\theta'_o}$ in these bases (recall that $\Sigma_{\theta_o}$ is block-upper triangular and $\Sigma_{\theta'_o}$ is block-lower triangular), then the matrix of $\varh_{1\ov1}^{\theta_o}$ is $\ov\Sigma_{\theta_o}$, that of $\varh_{1\ov1}^{\theta'_o}$ is $\ov\Sigma_{\theta'_o}$. Moreover, according to \ref{eq:11th}, $\varh$ is $\iota$-skew-Hermitian iff
\begin{equation}\label{eq:sigmatr}
\Sigma_{\theta'_o}=-{}^t\ov\Sigma_{\theta_o}.
\end{equation}

The results of Remark \ref{rem:imcan} can be read in $L_{\theta_o}$ via $a_1,a'_1:L_1\isom L_{\theta_o},L_{\theta'_o}$. We set $\can\addedo{_{\theta'_o}}\defin S_1^{-1}-S_2^{-1}:L_{\theta'_0}\to L_{\theta_o}$ and $F=\im\can\addedo{_{\theta'_o}}\subset L_{\theta_o}$. Then $\varh_{\theta'_o}:L_{\theta'_o}\otimes \ov{L_{\theta_o}}\to\kk$ induces $\hh_{\theta'_o}:F\otimes\ov F\to\kk$ by setting $\hh_{\theta'_o}(u,\ov v)=\varh_{\theta'_o}(x',\ov v)$ for some (or any) $x'\in L_{\theta'_o}$ such that $\can\addedo{_{\theta'_o}} x'=u$. As above, one checks that if $\varh$ is nondegenerate (\resp $\iota$-skew-Hermitian), then $\hh_{\theta'_o}$ is nondegenerate (\resp Hermitian) on $F$.

\addedm{On the other hand, the vector space $K_c$ is the intersection of the radical of $\Sigma_{\theta_o}+{}^t\ov\Sigma_{\theta_o}$ with $G_{c,1}$, and the matrix of $\varh_{K_c}$ is the conjugate of that of $S_{|K_c}$. If the splitting property at $c_o$ considered in Remark \ref{rem:Kc} is satisfied, and if we choose correspondingly bases of $G_{c_o,1}$ and $G_{c_o,2}$, the diagonal block $\Sigma_{\theta_o,c_oc_o}$ is itself block-diagonal with respect to this splitting, and the block $\Sigma_{\theta_o,K_{c_o}}$ is skew-adjoint.}

\begin{corollaire}\label{cor:posdefprime}
Assume $\kk=\RR$ or $\CC$. If $\varh$ is nondegenerate and $\iota$-skew-Hermitian, and if the Hermitian matrix $\Sigma_{\theta_o}+{}^t\ov\Sigma_{\theta_o}$ is positive semi-definite, then $\hh_{\theta'_o}$ is positive definite on $F$.
\end{corollaire}

\begin{proof}
Since $\hh_{\theta'_o}$ is nondegenerate, it is enough to show that $\hh_{\theta'_o}(u,\ov u)\geq0$ for all $u\in F$. Set $u=\can\addedo{_{\theta'_o}} x'$ and $x_1=a_1^{\prime-1}x'$. Then
\[
\hh_{\theta'_o}(u,\ov u)=\varh_{\theta'_o}(x',\ov{\can\addedo{_{\theta'_o}} x'})=\varh_{1\ov2}(a_1^{\prime-1}x',\ov{a_2^{-1}(S_1^{-1}-S_2^{-1})x'}).
\]
Now,
\[
a_2^{-1}S_1^{-1}x'=a_2^{-1}S_1^{-1}a'_1x_1=S_{\theta_o}x_1\quad\text{and}\quad
a_2^{-1}S_2^{-1}x'=a_2^{-1}S_2^{-1}a'_1x_1=S_{\theta'_o}x_1,
\]
hence $\hh_{\theta'_o}(u,\ov u)=\varh_{1\ov2}(x_1,\ov{(S_{\theta_o}-S_{\theta'_o})x_1})$. Since $\varh$ is $\iota$-skew-Hermitian, the matrix of $\varh_{1\ov2}(\cbbullet,\ov{(S_{\theta_o}-S_{\theta'_o})\cbbullet})$ is $\ov\Sigma_{\theta_o}+{}^t\Sigma_{\theta_o}$ after \eqref{eq:sigmatr}, which is positive semi-definite by assumption, hence $\hh_{\theta'_o}(u,\ov u)\geq0$.
\end{proof}

\subsection{Sesquilinear forms on the cohomology}\label{subsec:sesqcohom}
We now fix~$c\in\CC$ such that
\begin{enumeratea}
\item\label{assumpt:aa}
$\gr_c\cL=0$, and
\item\label{assumpt:b}
\deletedo{there exists $\theta_o\in S^1$ such that} $c_i\lethetao c$ for any $i$ (that is, all the~$c_i$ lie in an open half-plane with boundary going through~$c$).
\end{enumeratea}

Let us first compute the cohomology.

\begin{lemme}\label{lem:cohomS1}
We have $H^k(S^1,\cL_{\leq c})=0$ if $k\neq1$, $H^k(S^1,\cL/\cL_{\leq c})=0$ if $k\neq0$ and the exact sequence $0\to\cL_{\leq c}\to\cL\to\cL/\cL_{\leq c}\to0$ induces an exact sequence \addedo{(defining the morphism $\can$):}
\[
0\to H^0(S^1,\cL)\to H^0(S^1,\cL/\cL_{\leq c})\To{\can} H^1(S^1,\cL_{\leq c})\to H^1(S^1,\cL)\to 0.
\]
\end{lemme}

\begin{proof}
We compute the cohomology with the covering $(I_1,I_2)$. Then $H^k(I_1,\cL_{\leq c})=\nobreak0$ for any $k$ (and similarly for $I_2$): indeed, because of \eqref{assumpt:aa}, there is a Stokes direction in~$I_1$ for the pair $(c,c_i)$ for each~$i$, and according to the splitting given by Proposition \ref{prop:strict}\eqref{prop:strict1}, $\cL_{\leq c|I_1}$ decomposes as the direct sum of sheaves, each of which is constant on a proper open interval of~$I_1$ and $0$ on the complementary set, which is also connected; the assertion follows from the vanishing of $H^k_{(0,1)}([0,1),\kk)$ for any $k$. We also have $H^k(I_1\cap I_2,\cL_{\leq c})=0$ for $k\neq0$ and $H^0(I_1\cap I_2,\cL_{\leq c})=L_{\theta_o}$ after \eqref{assumpt:b}. We conclude that $H^1(S^1,\cL_{\leq c})=H^0(I_1\cap I_2,\cL_{\leq c})=L_{\theta_o}$.

Similarly, $H^k(I_j,\cL/\cL_{\leq c})=0$ ($j=1,2$) for $k\neq0$ follows from the similar vanishing of $H^k([0,1),\kk)$. We also have $H^k(I_1\cap I_2,\cL/\cL_{\leq c})=0$ for $k\neq0$ and $H^0(I_1\cap\nobreak I_2,\cL/\cL_{\leq c})\simeq L_{\theta'_o}$. Moreover, the restriction morphisms $H^0(I_j,\cL/\cL_{\leq c})\to H^0(I_1\cap\nobreak I_2,\cL/\cL_{\leq c})$ are isomorphisms. Therefore, the \v{C}ech complex
\[
H^0(I_1,\cL/\cL_{\leq c})\oplus H^0(I_2,\cL/\cL_{\leq c})\to H^0(I_1\cap I_2,\cL/\cL_{\leq c})
\]
has cohomology in degree $0$ at most.
\end{proof}

\begin{proposition}\label{prop:Hhc}
If~$c$ satisfies Assumptions \eqref{assumpt:aa} and \eqref{assumpt:b}, the natural pairing induced by~$\varh_c$\addedo{ from \eqref{eq:hc}}:
\begin{equation}\tag*{(\protect\ref{prop:Hhc})$(*)$}\label{eq:Hhc}
\varh_c:H^0(S^1,\cL/\cL_{\leq c})\otimes H^1(S^1,\iota^{-1}\ov\cL_{\leq c})\to H^1(S^1,\kk)=\kk
\end{equation}
is nondegenerate and corresponds to $\varh_{\theta'_o}$, via the isomorphisms
\begin{align*}
H^0(S^1,\cL/\cL_{\leq c})&\isom H^0(I_1\cap I_2,\cL/\cL_{\leq c})=L_{\theta'_o}\\
H^1(S^1,\iota^{-1}\ov\cL_{\leq c})&=H^0(I_1\cap I_2,\iota^{-1}\ov\cL_{\leq c})=\ov L_{\theta_o}.
\end{align*}
\end{proposition}

\begin{proof}
That the pairing~$\varh_c$ of \ref{eq:Hhc} is nondegenerate a priori follows from \eqref{eq:hcnondeg}. But this can also be obtained from the second part of the corollary, that we now prove with details. Let us consider the covering $(I_1,I_2)$ of $S^1$ with~$\theta_o$ as in \eqref{assumpt:b} above. As a consequence, $\cF\defin\cL/\cL_{\leq c}$ and $\cG\defin\iota^{-1}\ov\cL_{\leq c}$ are local systems in some neighbourhood of $\ov{I_1\cap I_2}$. Let us also denote by $\cC$ the constant sheaf $\kk_{S^1}$, and by $j_1:I_1\hto S^1$, $j_2:I_2\hto S^1$, $j_{12}:I_1\cap I_2\hto S^1$ the open inclusions. Given a sheaf $\ccF$ on $S^1$, we set $\ccF_a=j_{a,*}j_a^{-1}\ccF$ ($a=1,2,12$), $\ccF^0=\ccF_1\oplus\ccF_2$, and $\ccF^1=\ccF_{12}$. We have a Mayer-Vietoris complex
\[
\ccF^0\To{\delta}\ccF^1,\quad \delta(u_1,u_2)=u_1-u_2.
\]
The following is easy:
\begin{lemme}
Let $\ccF$ be a sheaf on $S^1$. If $\ccF$ is a local system in some neighbourhood of $\ov{I_1\cap I_2}$, then the Mayer-Vietoris complex is a resolution of $\ccF$ on $S^1$.\qed
\end{lemme}

We will apply this lemma to $\cF,\cG,\cC$. The simple complex $s(\cF^\cbbullet\otimes\cG^\cbbullet)$ is therefore a resolution of $\cF\otimes\cG$. The pairing $\varh_c:\cF\otimes\cG\to\cC$ extends as a morphism of complexes $\wt \varh:s(\cF^\cbbullet\otimes\cG^\cbbullet)\to\cC^\cbbullet$ as follows: we set
\begin{align*}
\wt \varh{}^0:\cF^0\otimes\cG^0=(\cF_1\oplus\cF_2)\otimes(\cG_1\oplus\cG_2)&\to\cC^0=(\cC_1\oplus\cC_2)\\
(u_1,u_2)\otimes(v_1,v_2)&\mto(\varh\addedo{_c}(u_1,v_1),\varh\addedo{_c}(u_2,v_2)),\\
\wt \varh{}^1:(\cF^0\otimes\cG^1)\oplus(\cF^1\otimes\cG^0)&\to\cC^1\\
\Big(\big[(u_1,u_2)\otimes v_{12}\big],\big[u_{12}\otimes(v_1,v_2)\big]\Big)&\mto\frac12\Big[\varh\addedo{_c}(u_1+u_2,v_{12})+\varh\addedo{_c}(u_{12},v_1+v_2)\Big],\\
\wt \varh{}^2&=0
\end{align*}
where we implicitly have extended~$\varh\addedo{_c}$ to pairings $\cF_a\otimes\cG_a\to\kk$, $a\in\{1,2,12\}$.

\begin{lemme}
The resolution $s(\cF^\cbbullet\otimes\cG^\cbbullet)$ of $\cF\otimes\cG$ is $\Gamma(S^1,\cbbullet)$-acyclic.
\end{lemme}

\begin{proof}
It is similar to that of Lemma \ref{lem:cohomS1}.
\end{proof}

Clearly, $\cC^\cbbullet$ is also $\Gamma(S^1,\cbbullet)$-acyclic. As a consequence (\cf \cite[Th\ptbl II.4.7.2]{Godement64}), the morphism~$\varh_c$ is expressed by taking $H^1$ of the morphism of complexes
\[
s\Big(\Gamma(S^1,\cF^\cbbullet)\otimes\Gamma(S^1,\cG^\cbbullet)\Big)\to\Gamma\big(S^1,s(\cF^\cbbullet\otimes\cG^\cbbullet)\big)\To{\Gamma(S^1,\wt \varh)}\Gamma(S^1,\cC^\cbbullet).
\]
Using that $\Gamma(S^1,\cG^0)=0$, after Lemma \ref{lem:cohomS1}, we regard~$\varh_c$ as the composition
\[
\xymatrix{
\Gamma(S^1,\cF^0)\otimes\Gamma(S^1,\cG^1)\ar[r]&\Gamma(S^1,\cF^0\otimes\cG^1)\ar[r]^-{\wt \varh{}^1}&\Gamma(S^1,\cC^1)\ar[d]\\
H^0(S^1,\cF)\otimes H^1(S^1,\cG)\ar[u]&&H^1(S^1,\kk)
}
\]
Let $(u,u)\in H^0(S^1,\cF)\subset \Gamma(S^1,\cF_1)\oplus\Gamma(S^1,\cF_2)=L_{\theta'_o}\oplus L_{\theta'_o}$ and let $v\in H^1(S^1,\cG)=\Gamma(I_1\cap I_2,\cG)=\ov L_{\theta_o}$. Using the previous formula for $\wt \varh{}^1$, $(u,u)\otimes v$ is sent to $\varh_{\theta'_o}(u,v)$ in the component $\kk_{\theta'_o}$ of $\Gamma(I_1\cap I_2,\kk)$ and to $0$ in the component $\kk_{\theta_o}$. The second assertion of \replacedo{Proposition}{Corollary} \ref{prop:Hhc} follows.
\end{proof}

\subsection{A Hermitian pairing on the cohomology}
We continue to assume that~$c$ satisfies Assumptions \eqref{assumpt:aa} and \eqref{assumpt:b} of \S\ref{subsec:sesqcohom}. Let us first make explicit the middle morphism in the exact sequence of Lemma \ref{lem:cohomS1}.

\begin{lemme}\label{lem:imcan}
Through the natural identifications $H^0(S^1,\cL/\cL_{\leq c})\isom L_{\theta'_o}$ and $H^1(S^1,\cL_{\leq c})\isom L_{\theta_o}$, the natural morphism $\can:H^0(S^1,\cL/\cL_{\leq c})\to H^1(S^1,\cL_{\leq c})$ is identified with $\addedo{\can_{\theta'_o}={}}S_1^{-1}-S_2^{-1}:L_{\theta'_o}\to L_{\theta_o}$.
\end{lemme}

\begin{proof}
\addedo{Applying the snake lemma, }we obtain the exact sequence of Lemma \ref{lem:cohomS1} from the following exact sequence of (vertical) Mayer-Vietoris complexes\addedo{ which computes the cohomology of the corresponding sheaves, where the vertical arrows are the differences $\rho_1-\rho_2$ of the natural restriction morphisms from $I_1$ or $I_2$ to $I_1\cap I_2$}:
\[
\xymatrix@C=.2cm{
0\ar[r]&0\ar[r]\ar[d]&\Gamma(I_1,\cL)\oplus\Gamma(I_2,\cL)\ar[r]\ar[d]&\Gamma(I_1,\cL/\cL_{\leq c})\oplus\Gamma(I_2,\cL/\cL_{\leq c})\ar[r]\ar[d]&0\\
0\ar[r]&\Gamma(I_1\cap I_2,\cL_{\leq c})\ar[r]&\Gamma(I_1\cap I_2,\cL)\ar[r]&\Gamma(I_1\cap I_2,\cL/\cL_{\leq c})\ar[r]&0
}
\]
Given $u\in \Gamma(I_1,\cL/\cL_{\leq c})\simeq L_{\theta'_o}$, its lift in $\Gamma(I_1,\cL)\simeq L_{\theta_o}$ is $S_1^{-1}u$. Then $(u,u)\in L_{\theta'_o}\oplus L_{\theta'_o}$ is lifted as $(S_1^{-1}u,S_2^{-1}u)$, and its image in $\Gamma(I_1\cap I_2,\cL)\simeq L_{\theta'_o}\oplus L_{\theta_o}$ is $(\addedo{0},(S_1^{-1}-S_2^{-1})u)$.
\end{proof}

Let $F_c=\im\can\subset H^1(S^1,\cL_{\leq c})$. According to \replacedo{Proposition}{Corollary} \ref{prop:Hhc}, Lemma \ref{lem:imcan} and Remark \ref{rem:imcan}, the sesquilinear pairing $\varh_c$, as defined by \ref{eq:Hhc}, induces a sesquilinear pairing
\begin{equation}\label{eq:hh}
\hh_c: F_c\otimes\ov F_c\to\kk
\end{equation}
by setting $\hh_c(u,\ov v)\defin \varh_c(x',\ov v)$ for some (or any) $x'\in H^{\addedo{0}}(S^1,\cL/\cL_{\leq c})$ such that $\can x'=u$. Moreover, if $\varh$ is nondegenerate (\resp $\iota$-skew-Hermitian), then $\hh_c$ is nondegenerate (\resp Hermitian) on $F_c$. From Corollary \ref{cor:posdefprime} we get:

\begin{corollaire}\label{cor:posdef}
Assume $\kk=\CC$ and the involution is the conjugation, or $\kk=\RR$ or~$\QQ$ and the involution is the identity. If the invertible matrix $\Sigma_{\theta_o}$ is such that the Hermitian matrix $\Sigma_{\theta_o}+{}^t\ov\Sigma_{\theta_o}$ is positive semi-definite, then $\hh_c$ is positive definite on~$F\addedo{_c}$.\qed
\end{corollaire}

\section{Minimal constructible sheaves on $\PP^1$ with Stokes structure at infinity}\label{sec:pervStokesP1}

In this section, we set $X=\Afu\cup S^1_\infty$ (with respect to the setting of \S\ref{subsec:LaplaceD}, the coordinate on $\Afu$ should be denoted by $\tau$). The inclusions are denoted by $j_\infty:\Afu\hto X$ and $i_\infty:S^1_\infty\hto X$ and the projection $X\to\PP^1$ by $\varpi$.

Let $\cF$ be a constructible sheaf on $\Afu$ with finite singularity set $\Sigma$. Its extension $j_{\infty,*}\cF$ is a sheaf on $X$ whose restriction to $X\moins\Sigma$ (hence also to $S^1_\infty$) is a local system.

\begin{definition}\label{def:consStokes}
By a Stokes structure at infinity $(\cF,\cF_\bbullet)$ on $\cF$ we will mean the data of a family of subsheaves $\cF_{\leq c}$ ($c\in\CC$) of $j_{\infty,*}\cF$ such that
\begin{enumerate}
\item
for each $c\in\CC$, $j_\infty^{-1}\cF_{\leq c}=\cF$,
\item
the family $\cL_\bbullet\defin i_\infty^{-1}\cF_\bbullet$ of subsheaves of $\cL\defin i_\infty^{-1}j_{\infty,*}\cF$ is a Stokes filtration of the local system $\cL$ as in \S\ref{subsec:Stokesfil}.
\end{enumerate}
We also say that $(\cF,\cF_\bbullet)$ is a Stokes-filtered constructible sheaf on $\Afu$, \replacedo{meaning}{understating} that the Stokes filtration is at infinity.
\end{definition}

Recall that a sheaf on $X$ can be defined through its restrictions to $\Afu$ and $S^1_\infty$ and gluing data. In such a way, the inclusion $\cL_{<c}\hto\cL_{\leq c}$ determines a unique subsheaf $\cF_{<c}$ of $\cF_{\leq c}$ whose restriction to $\Afu$ is $\cF$ and that to $S^1_\infty$ is $\cL_{<c}$.

We define in this way a category, for which the morphisms are morphisms of sheaves $\lambda:\cF\to\cF'$ such that $i_\infty^{-1}j_{\infty,*}\lambda$ is a morphism of Stokes-filtered local systems. As a consequence of Proposition \ref{prop:catStokesfiltabelian}, this category is abelian.

\begin{lemme}\label{lem:iinfty!}
Let $(\cF,\cF_\bbullet)$ be a Stokes-filtered constructible sheaf. Then for each $c\in\CC$ the complex $i_\infty^!\cF_{\leq c}$ has cohomology in degree $1$ at most and $\cH^1(i_\infty^!\cF_{\leq c})=\cL/\cL_{\leq c}$. A similar assertion holds for $\cF_{<c}$.
\end{lemme}

\begin{proof}
We have $\bR j_{\infty,*}\cF=j_{\infty,*}\cF$ and the distinguished triangle
\[
i_\infty^{-1}\cF_{\leq c}\to i_\infty^{-1}j_{\infty,*}\cF\to i_\infty^!\cF_{\leq c}[1]\To{+1}
\]
reduces to the exact sequence
\[
0\to\cL_{\leq c}\to\cL\to\cL/\cL_{\leq c}\to0,
\]
showing that $i_\infty^!\cF_{\leq c}[1]$ has cohomology in degree $0$ at most, this cohomology being equal to $\cL/\cL_{\leq c}$.
\end{proof}

In the following, we only consider constructible sheaves $\cF$ for which the singularity set $\Sigma$ is reduced to $\{\tau=0\}$. We denote by $j_0$ the inclusion $\Afu\moins\{0\}\hto\Afu$. We say that $\cF$ or $(\cF,\cF_\bbullet)$ is \emph{minimal} (or middle extension) if $\cF=j_{0,*}j_0^{-1}\cF$ (by our assumption, $j_0^{-1}\cF$ is a locally constant sheaf on $\Afu\moins\{0\}$).

\addedo{It should be noted that minimality at $\tau=0$, as considered here, is a~priori not related to the minimality property of the Stokes filtration at $\tau=\infty$, as in Definition \ref{def:critminext} and Lemma \ref{lem:critminext}. At the level of $\Clt$ and $\Cltau$-modules considered in \S\ref{subsec:LaplaceD}, the latter is related to the property that~$M$ is a minimal extension at its singularities at finite distance, while the former is related to the property that $N$ is a minimal extension at $\tau=0$.}

\begin{lemme}\label{lem:equiv}
Given a Stokes-filtered local system $(\cL,\cL_\bbullet)$ on $S^1_\infty$, there exists a unique (up to unique isomorphism) minimal Stokes-filtered constructible sheaf $(\cF,\cF_\bbullet)$ such that $(\cL,\cL_\bbullet)=(i_\infty^{-1}j_{\infty,*}\cF,i_\infty^{-1}\cF_\bbullet)$.
\end{lemme}

\begin{proof}
For the existence, let us denote by $\pi:\Afu\moins\{0\}\to S^1_\infty$ the projection (quotient by $\RR_+^*$). Set $\cF^*=\pi^{-1}\cL$ and $\cF=j_{0*}\cF^*$. Then $i_\infty^{-1}j_{\infty,*}\cF=\cL$ and the inclusion $\cL_\bbullet\hto\cL$ determines a unique subsheaf $\cF_\bbullet$ of $j_{\infty,*}\cF$ such that $i_\infty^{-1}\cF_\bbullet=\cL_\bbullet$ and $j_\infty^{-1}\cF_\bbullet=\cF$.

Given two such objects $(\cF,\cF_\bbullet)$ and $(\cF',\cF'_\bbullet)$, the identity morphism $\cL=\cL$ extends in a unique way as an isomorphism $\cF^*\simeq\cF^{\prime*}$ and then as an isomorphism $j_{0,*}\cF^*\simeq j_{0,*}\cF^{\prime*}$, proving the uniqueness. The uniqueness of the isomorphism inducing the identity $\cL=\cL$ is also clear.
\end{proof}

Concerning the compatibility, through the equivalence of Lemma \ref{lem:equiv}, of the operations considered in \S\ref{sec:operations}, let us notice that compatibility with \deletedo{twist and} involution~$\iota$ is straightforward. For the duality, we will check it now. Before doing so, notice that, for each $c\in\CC$, we have a natural morphism $j_{\infty,!}\cF\to\cF_{\leq c}$ which induces, after applying $i_\infty^!$ and after Lemma \ref{lem:iinfty!}, the surjection $\cL\to\cL/\cL_{\leq c}$.

Recall that the dualizing complex on $X$ is $j_{\infty,!}\kk_{\Afu}[2]$. If $\cG$ is a sheaf on $X$, we denote by $\DD(\cG)=\bR\cHom(\cG,j_{\infty,!}\kk_{\Afu}[2])$ its Poincaré-Verdier dual, and by $\DD'(\cG)=\bR\cHom(\cG,j_{\infty,!}\kk_{\Afu})$ the shifted complex. As in \S\ref{subsec:duality}, we say that $\DD'(\cG)$ is a sheaf to mean that the complex $\DD'(\cG)$ has cohomology in degree $0$ at most. In such a case, we identify $\DD'(\cG)$ with the sheaf $\cHom(\cG,j_{\infty,!}\kk_{\Afu})$.

Note that, on $\Afu$, if $\cF$ is a minimal constructible sheaf as above, $\DD'\cF$ is a sheaf, which is constructible with singularity at $0$ at most, and is minimal. We will denote it by $\cF^\vee$.

\begin{proposition}[Duality]\label{prop:duality}
The category of minimal Stokes-filtered constructible sheaves is stable by Poincaré-Verdier duality (up to a shift by $2$). More precisely, for each object $(\cF,\cF_\bbullet)$,
\begin{enumerate}
\item\label{prop:duality1}
$\DD'(j_{\infty,!}\cF)$, $\DD'(\cF_{\leq c})$ and $\DD'(\cF_{<c})$ are sheaves for each $c\in\CC$,
\item\label{prop:duality2}
the dual $\DD'(\cF_{\leq c})\to\DD'(j_{\infty,!}\cF)=j_{\infty,*}\cF^\vee$ of the natural morphism $j_{\infty,!}\cF\to\cF_{\leq c}$ is injective for each $c\in\CC$, and similarly for $\cF_{<c}$,
\item\label{prop:duality3}
the family $(\cF^\vee_{\leq c})_{c\in\CC}$ of subsheaves of $j_{\infty,*}\cF^\vee$ defined by $\cF^\vee_{\leq c}=\DD'(\cF_{<-c})$ is a Stokes filtration at infinity of $\cF^\vee$, for which $\cF^\vee_{<c}=\DD'(\cF_{\leq -c})$.
\end{enumerate}
\end{proposition}

\begin{proof}
\addedo{On $\Afu$, the first assertion is equivalent to saying that $\DD'\cF$ is a sheaf, and this has been noticed before the proposition. It is therefore enough to prove that $i_\infty^{-1}\DD'(j_{\infty,!}\cF)$, etc. are sheaves on $S^1_\infty$, because $i_\infty^{-1}$ commutes with taking cohomology sheaves.}

\replacedo{It}{Its} is classical that $\DD j_{\infty,!}\cF=\bR j_{\infty,*}\DD\cF$, and hence $\DD' j_{\infty,!}\cF=\bR j_{\infty,*}\DD'\cF=j_{\infty,*}\cF^\vee$, \addedo{so $i_\infty^{-1}\DD'(j_{\infty,!}\cF)=i_\infty^{-1}j_{\infty,*}\cF^\vee$ is a sheaf.}

For each $c\in\CC$, we have (\cf \cite[Prop\ptbl3.1.13]{K-S90})
\begin{equation}\label{eq:DDi}
\begin{split}
i_\infty^!\DD'(\cF_{\leq c})&=\bR\cHom(i_\infty^{-1}\cF_{\leq c},i_\infty^!j_{\infty,!}\kk_{\Afu})\\
&=\bR\cHom(i_\infty^{-1}\cF_{\leq c},\kk_{S^1_\infty}[-1])=:\DD'(i_\infty^{-1}\cF_{\leq c})[-1]
\end{split}
\end{equation}
and by biduality, we have $\DD'(i_\infty^!\cF_{\leq c}[1])=i_\infty^{-1}\DD'(\cF_{\leq c})$. Lemmas \ref{lem:dualS1} and \ref{lem:iinfty!} show that $\DD'(i_\infty^!\cF_{\leq c}[1])=\DD'(\cL/\cL_{\leq c})$ and $\DD'(i_\infty^{-1}\cF_{\leq c})=\DD'(\cL_{\leq c})$ \replacedo{are sheaves}{have cohomology in degree~$0$ at most}. \addedo{The first property implies then that $i_\infty^{-1}\DD'(\cF_{\leq c})$ is a sheaf.}

Arguing similarly for $\cF_{<c}$, we obtain \replacedo{that $i_\infty^{-1}\DD'(\cF_{<c})$ is a sheaf}{the remaining part of \ref{prop:duality}\eqref{prop:duality1}}.

The assertion \ref{prop:duality}\eqref{prop:duality2} needs only be checked \addedo{on} $S^1_\infty$, and by duality, according to \eqref{eq:DDi}, it amounts to proving that $\cH^1(i_\infty^!j_{\infty,!}\cF)\to\cH^1(i_\infty^!\cF_{\leq c})$ is onto. This follows from Lemma \ref{lem:iinfty!}\addedo{ and the paragraph after the proof of Lemma \ref{lem:equiv}}.

Now, \ref{prop:duality}\eqref{prop:duality3} follows from Lemma \ref{lem:dualS1}.
\end{proof}

Let~$\varh$ be a $\iota$-sesquilinear form on $\cF$, that is, a pairing $\cF\otimes\iota^{-1}\ov\cF\to\kk_{\Afu}$. It extends in a unique way as a nondegenerate $\iota$-sesquilinear pairing $\varh:j_{\infty,!}\cF\otimes \iota^{-1}\ov{j_{\infty,*}\cF}\to j_{\infty,!}\kk_{\addedo{\Afu}}$ and defines a morphism $j_{\infty,!}\cF\to\DD'(j_{\infty,*}\ov\cF)$. Arguing as for \eqref{eq:DDi}, it induces a morphism $i_\infty^!j_{\infty,!}\cF[1]\simeq\cL\to\DD'(i_\infty^{-1}j_{\infty,*}\ov\cF)=\ov\cL{}^\vee$, and thus a sesquilinear pairing $\varh_\infty:\cL\otimes\ov\cL\to\kk_{S^1_\infty}$. \replacedo{Arguing as in Proposition \ref{prop:duality}, one checks that if $\varh$ is nondegenerate, \ie $\cF\to\DD'(\ov\cF)$ is an isomorphism, then so is $\varh_\infty$. Conversely, arguing as in Lemma \ref{lem:equiv}, one reconstructs $\varh$ from $\varh_\infty$ and obtains the non-degeneracy of~$\varh$ from that of $\varh_\infty$}{Arguing as in Lemma \ref{lem:equiv} one gets that $\varh$ is nondegenerate, \ie $\cF\to\DD'(\ov\cF)$ is an isomorphism, if and only if $\varh_\infty$ is so}.

Let us now express in terms of $(\cF,\cF_\bbullet)$ the compatibility of $\varh_\infty$ with the Stokes filtration. Extend $\varh$ as a pairing $j_{\infty,*}\varh:j_{\infty,*}\cF\otimes \iota^{-1}\ov{j_{\infty,*}\cF}\to j_{\infty,*}\kk_X$. This pairing induces for each $c\in\CC$ a pairing $\varh_c:\cF_{<c}\otimes\iota^{-1}(\ov\cF_{\leq c})\to j_{\infty,*}\kk_X$.

\begin{lemme}\label{lem:varhs}
The pairing $\varh_\infty$ is compatible with the Stokes filtration if and only if, for each $c\in\CC$, the pairing $\varh_c$ takes values in $j_{\infty,!}\kk_X$. When such is the case, the induced pairing $i_\infty^!\cF_{<c}[1]\otimes i_\infty^{-1}\iota^{-1}(\ov\cF_{\leq c})\to\kk_{S^1_\infty}$ is identified with $\varh_{\infty,c}:(\cL/\cL_{<c})\otimes\nobreak\iota^{-1}(\ov\cL_{\leq c})\to\kk_{S^1_\infty}$.
\end{lemme}

\begin{proof}
We first note that the pairing $i_\infty^{-1}j_{\infty,*}\varh:\cL\otimes\ov\cL\to\kk_{S^1_\infty}$ coincides with $\varh_\infty$ through the natural isomorphisms $i_\infty^{-1}j_{\infty,*}\cF\to i_\infty^!j_{\infty,!}\cF[1]$ and similarly for $\kk_X$. The condition on $\varh_c$ is then equivalent to the vanishing of $\varh_\infty$ restricted to $\cL_{<c}\otimes\iota^{-1}(\ov\cL_{\leq c})$ for each $c$, and this is equivalent to the compatibility with the Stokes filtration. The second part of the lemma follows from \eqref{eq:DDi} and Lemma \ref{lem:dualS1}.
\end{proof}

When the condition of the lemma is fulfilled, we say that $\varh$ is a $\iota$-sesquilinear form on $(\cF,\cF_\bbullet)$. We say that it is nondegenerate if it is nondegenerate on $\cF$.

\begin{proposition}[Sesquilinear pairing on cohomology]\label{prop:pairing}
Let~$\varh$ be a nondegenerate $\iota$\nobreakdash-ses\-qui\-li\-near form on $(\cF,\cF_\bbullet)$. Let us choose $c\in\CC$ satisfying the assumptions \eqref{assumpt:aa} and \eqref{assumpt:b} of \S\ref{subsec:sesqcohom}, so that in particular $\cF_{\leq c}=\cF_{<c}$ and $\cF^\vee_{\leq -c}=\DD'(\cF_{\leq c})$. Then the following properties hold:
\begin{enumerate}
\item\label{prop:pairing1}
$H^k(X,\cF_{\leq c})=0$ for $k\neq1$ and, via the natural restriction morphism, $H^1(X,\cF_{\leq c})$ is identified with $F_c\defin\im\can\subset H^1(S^1_\infty,\cL_{\leq c})$ (\cf Lemma \ref{lem:imcan}).
\item\label{prop:pairing2}
The sesquilinear pairing $\varh_c:\cF_{<c}\otimes_{\kk}\iota^{-1}(\ov\cF_{\leq c})\to j_{\infty,!}\kk_{\Afu}$ induces a sesquilinear pairing $\wh \varh_c:H^1(X,\cF_{\leq c})\otimes_{\kk} \ov{H^1(X,\cF_{\leq c})}\to\kk$ which is identified with $\hh_c$, defined by \eqref{eq:hh}, via the identification of \eqref{prop:pairing1}.
\end{enumerate}
\end{proposition}

In \eqref{prop:pairing2}, we use the canonical isomorphism $H^1(X,\ov\cF_{\leq c})\simeq H^1(X,\iota^{-1}(\ov\cF_{\leq c}))$.

\begin{proof}[Proof of Proposition \ref{prop:pairing}\eqref{prop:pairing1}]
Let us fix some notation. We denote by $e:Y\to X$ the real blow-up of the origin in $X$ and we set $S^1_0=e^{-1}(0)$, so that $Y=S^1\times[0,\infty]$. We consider the covering $Y=U_1\cup U_2$, with $U_k=I_k\times\nobreak[0,\infty]$, $k=1,2$, where $I_1,I_2$ are as in \S\ref{subsec:Stokesdata}.

We set $X^*=X\moins\{0\}=Y^*=Y\moins S^1_0$ and we denote by $j$ the inclusion $X^*=Y^*\hto Y$. We also set $\cF_{\leq c}^*=\addedo{\cF_{\leq c|X^*}}$ and we define $\cL^*$ as the pull-back of $\cL$ by the projection $Y^*\to S^1_\infty$.

Set $\cG=j_*\cF_{\leq c}^*$ and $\cG'=j_!\cF_{\leq c}^*$. We have $\cG=\bR j_*\cF_{\leq c}^*$ and $\cF_{\leq c}=e_*\cG$ (distinct from $\bR e_*\cG$ if the monodromy of~$\cL$ admits $1$ as an eigenvalue). Let us denote by $j_1:U_1\hto Y$ the open inclusion, and by $\cG_1$ the complex $\bR j_{1,*}(\cG_{|U_1})$. We similarly use the notation $\cG_2$ and $\cG_{12}$. It will be convenient to denote by $\cL_Y$ the pull-back of $\cL$ by the projection $Y\to S^1_\infty$.

\begin{lemme}\label{lem:annul}
Under Assumptions \eqref{assumpt:aa} and \eqref{assumpt:b} of \S\ref{subsec:sesqcohom}, the complex $\cG_1$ is equal to the sheaf $j_{1,*}(\cG_{|U_1})$. Moreover, $H^k(U_1,\cG\addedo{_{|U_1}})=H^k(Y,\cG_1)=0$ for any $k$. A similar assertion holds for $\cG_2$, and for $\cG_{12}$ if $k\neq0$. Moreover, $H^k(Y,\cG)=0$ if $k\neq\addedo{1}$ and $H^1(Y,\cG)=H^0(U_1\cap\nobreak U_2,\cG)=H^0(Y,\cG_{12})\simeq L_{\theta_o}$.
\end{lemme}

\begin{proof}
\deletedn{Whenever $\cG$ is a local system in the neighbourhood of a point in $\partial I_1\times[0,\infty]$, the first assertion is easy at this point. By using the local splitting of the Stokes filtration, it then remains to check the assertion at $(\theta'_o+\epsilon,\infty)$, where $\cG$ is the extension by $0$ to $(\theta'_o-\epsilon,\theta'_o+\nobreak\epsilon)\times\{\infty\}$ of a local system on $(\theta'_o-\epsilon,\theta'_o+\nobreak\epsilon)\times(\eta,\infty)$, with $\eta$ large. }By assumption, $I_1$ is the interval $(\theta_o-\epsilon,\theta'_o+\epsilon)$. We need to check that $\cG_1$ is a sheaf along $\partial I_1\times[0,\infty]$. Note that $\cG$ is a local system on $Y\moins S^1_\infty$, so the assertion is clear along $\partial I_1\times[0,\infty)$. Now the assertion is local near the points $(\theta_o-\epsilon,\infty)$ and $(\theta'_o+\epsilon,\infty)$, and we can use a local splitting of the Stokes filtration to reduce to the case where $\cG$ is a local system in the neighbourhood of $(\theta_o-\epsilon,\infty)$, which is already treated, or $\cG$ is the extension by zero of a (constant) local system on an open set like $(\theta'_o-\epsilon,\theta'_o+\epsilon)\times(\eta,\infty)$, with $\eta\gg0$, via the inclusion $(\theta'_o-\epsilon,\theta'_o+\epsilon)\times(\eta,\infty)\hto(\theta'_o-\epsilon,\theta'_o+\epsilon)\times(\eta,\infty]$ (by our assumption on \replacedo{$\theta_o$ and $c$}{$c$ and~$\theta_o$}, this \deletedo{can} occur\addedo{s} only at~$\theta'_o$).

We are thus reduced to showing, since the local system is constant on this neighbourhood, and retracting $(\theta'_o-\epsilon,\theta'_o+\epsilon)\times(\eta,\infty]$ to $\{\theta'_o\}\times(\eta,\infty]$, that $H^1_{\{\infty\}}((\eta,\infty],\CC)=0$, which is clear.

Let us show
\begin{equation}\label{eq:annulU1}
H^k(U_1,\cG\addedo{_{|U_1}})=0\quad\forall k.
\end{equation}
According to \eqref{eq:splittings}, we can choose on $U_1$ an isomorphism $\cL\addedo{_{Y|U_1}}\simeq\bigoplus_{i=1}^r\addedo{\cG_{c_i}^{(1)}}$, where~$\addedo{\cG_{c_i}^{(1)}}$ are local systems on $U_1$, and the isomorphism is compatible with the Stokes filtration on $I_1\times\{\infty\}$, so we can work independently with each summand~$\addedo{\cG_{c_i}^{(1)}}$. \addedo{Arguing as for $\bR j_{1,*}\cG_{|U_1}$ above, we find that each $\bR j_{1,*}\cG_{c_i}^{(1)}$ is a sheaf $j_{1,*}\cG_{c_i}^{(1)}$ and therefore $H^k(U_1,\cG_{c_i}^{(1)})=H^k(\ov U_1,j_{1,*}\cG_{c_i}^{(1)})$.}

\addedo{Arguing as for $\cG_1$, we find that $j_{1,*}\cG_{c_i}^{(1)}$ is a (constant) local system on~$\ov U_1\moins\big([\theta_i,\theta'_o+\epsilon]\times\{\infty\}\big)$, for some $\theta_i\in I_1$, and is zero on $[\theta_i,\theta'_o+\epsilon]\times\{\infty\}$. Identifying topologically the closure $\ov U_1$ of $U_1$ with a closed disc $\ov D$, the cohomology of such a sheaf is the relative cohomology modulo a closed interval in $\partial\ov D$ of the constant sheaf on $\ov D$}, so identically $0$, hence \eqref{eq:annulU1}. The same result holds for $U_2$, of course.

Let us now compute $H^k(U_1\cap U_2,\cG)$. We identify each connected component of $\ov{U_1\cap U_2}$ to a closed disc, and a similar computation shows that \hbox{$H^k(U_1\cap U_2,\cG)=0$} for $k\neq0$ and $\Gamma(U_1\cap U_2,\cG)\simeq L_{\theta_o}$.
\end{proof}

Recall that $\cF_{\leq c}=e_*\cG$. On the other hand, $\bR^1e_*\cG$ is a sheaf supported at the origin on $X$, whose germ is equal to $H^1(S^1_0,\cL\addedo{_{Y|S^1_0}})$. \addedo{Since $\bR e_*\cG$ has only two cohomology sheaves, there is a natural triangle
\[
e_*\cG\to\bR e_*\cG\to\bR^1e_*\cG[-1]\To{+1}
\]
inducing a long exact sequence in hypercohomology over $X$. Note that the space $\bH^k(X,\bR^1e_*\cG[-1])=H^{k-1}(X,\bR^1e_*\cG)$ is equal to~$0$ if $k\neq1$ and to the germ $(\bR^1e_*\cG)_0=H^1(S^1_0,\cL_{Y|S^1_0})$ if $k=1$. On the other hand, $\bH^k(X,e_*\cG)=H^k(X,e_*\cG)$ and $\bH^k(X,\bR e_*\cG)=H^k(Y,\cG)$. Moreover, from the previous lemma we get $H^2(Y,\cG)=0$. We therefore obtain a long exact sequence}
\begin{equation}\label{eq:HG}
0\to H^1(X,\cF_{\leq c})\to H^1(Y,\cG)\to H^1(S^1_0,\cL_Y)\to H^2(X,\cF_{\leq c})\to0,
\end{equation}
where the middle map is the restriction morphism to $S^1_0$. We have a commutative diagram
\[
\xymatrix{
H^1(Y,\cG)\ar[d]\ar[r]&H^1(Y,j_{\infty,*}j_\infty^{-1}\cG)\ar[d]^\wr\ar[r]^-\sim&H^1(S^1_0,\cL_Y)\\
H^1(S^1_\infty,\cL_{\leq c})\ar[r]&H^1(S^1_\infty,\cL)
}
\]
where the vertical arrows are the restriction to $S^1_\infty$, and the lower horizontal line is the right part of the exact sequence in Lemma \ref{lem:cohomS1}. Moreover, the left vertical morphism is an isomorphism, according to the computation of Lemma \ref{lem:annul}. As a consequence, $H^1(Y,\cG)\to H^1(S^1_0,\cL_Y)$ is onto and its kernel $H^1(X,\cF_{\leq c})$ is identified with $F_c=\im\can$ via the restriction morphism $H^1(X,\cF_{\leq c})\to H^1(S^1_\infty,\cL_{\leq c})$.
\end{proof}

\begin{proof}[Proof of Proposition \ref{prop:pairing}\eqref{prop:pairing2}]
We use the commutative diagram
\[
\xymatrix{
H^1(X,\cF_{\leq c})\otimes H^1(X,\iota^{-1}\ov\cF_{\leq c})\ar[r]^-{\wh \varh_c}\ar@<1cm>[d]& H^2(X,j_{\infty,!}\kk_{\Afu})\\
H^1(S^1_\infty,i_\infty^!\cF_{\leq c})\otimes H^1(S^1_\infty,i_\infty^{-1}\iota^{-1}\ov\cF_{\leq c})\ar@<1.5cm>[u]^{\can}\ar[r]^-{\varh_{\infty,c}}&H^1(S^1_\infty,\kk_{S^1})\ar[u]
}
\]
where the vertical morphisms are induced by the restriction or by the natural morphism $\bR i_{\infty,*}i_\infty^!\to\id$ and we can eliminate $\iota^{-1}$ in the cohomology. The identification of the lower pairing $\varh_{\infty,c}$ to~$\varh_c$ of \ref{eq:Hhc} follows from Lemma \ref{lem:iinfty!}. To conclude, we use \eqref{eq:hh}.
\end{proof}

\section{Riemann-Hilbert correspondence and sesquilinear pairings}\label{sec:RHsesqui}
All along this section, we will only consider holonomic $\Cltau$-modules $N$ (where~$\tau$ is the coordinate on the affine line $\Afu=\Afu_\tau$) with a regular singularity at $\tau=0$ and no other singularities at finite distance, and of exponential type at $\tau=\infty$, meaning that the Laplace (or inverse Laplace) transform has only regular singularities (\cf\eg\cite[Lemma 1.5]{Bibi08}). \deletedo{Lastly, assume (for the sake of simplicity) that~$N$ is a minimal extension at $\tau=0$.} For such a holonomic $\Cltau$-module $N$, we will denote by $\cN$ the $\cO_{\PP^1}(*\infty)$-module with connection associated with~$N$. We will use the notation of \S\ref{sec:pervStokesP1}.

\subsection{The Riemann-Hilbert correspondence}\label{subsec:RH}
Denote by $\cA^{\rmod\infty}_X$ the sheaf on~$X$ of holomorphic functions on~$\Afu$ which have moderate growth along $S^1_\infty$. This is naturally a subsheaf of $j_{\infty,*}\cO_{\Afu}$. It is a $\varpi^{-1}\cO_{\PP^1}(*\infty)$-submodule and is stable by the natural action of $\varpi^{-1}\cD_{\PP^1}$ on $j_{\infty,*}\cO_{\Afu}$. We will also consider the subsheaves $e^{c\tau}\cA^{\rmod\infty}_X\subset j_{\infty,*}\cO_{\Afu}$ ($c\in\CC$), which satisfy similar properties and coincide with $\cO_{\Afu}$ on $\Afu$. Given $\addedo{e^{i\theta}}\in S^1_\infty$, the germs satisfy $e^{c\tau}\cA^{\rmod\infty}_{X,\theta}\subset e^{c'\tau}\cA^{\rmod\infty}_{X,\theta}$ as soon as $c\leqtheta c'$ (\cf\S\ref{subsec:Stokesfil}\eqref{enum:Stokesfil1}), since $e^{c\tau}=e^{c'\tau}\cdot e^{(c-c')\tau}$ and $e^{(c-c')\tau}$ has moderate growth as well as all its derivatives near~$\theta$.

For each $c\in\CC$ we denote by $\DR_{\leq c}(N)$ the complex
\begin{equation}\label{eq:DRleqc}
(e^{c\tau}\cA^{\rmod\infty}_X)\otimes_{\cO_{\PP^1}(*\infty)}\cN
\To{\nabla}(e^{c\tau}\cA^{\rmod\infty}_X)\otimes_{\cO_{\PP^1}(*\infty)}(\Omega^1_{\PP^1}\otimes\cN).
\end{equation}
(The complex $\DR_{\leq 0}(N)$ is also denoted by $\DR^{\rmod\infty}(N)$.) We have a natural identification
\begin{equation}\label{eq:identificationc}
j_\infty^{-1}\DR_{\leq c}(N)=\DR^\an N.
\end{equation}
If we denote by $\cE^{c\tau}$ the $\Cltau$-module $(\CC[\tau],\partial_\tau+c)$, the termwise multiplication by $e^{-c\tau}$ induces an isomorphism $\DR_{\leq c}(N)\isom\DR^{\rmod\infty}(\cE^{c\tau}\otimes N)$.

There is a rapid-decay analogue. Firstly, the subsheaf $\cA^{\rd\infty}_X\subset\cA^{\rmod\infty}_X$ consists of those functions which have rapid decay along $S^1_\infty$. Then $\DR_{<c}(N)$ is defined by a complex similar to \eqref{eq:DRleqc} where we replace $\cA^{\rmod\infty}_X$ with $\cA^{\rd\infty}_X$\addedo{ (the complex $\DR_{<0}(N)$ is also denoted by $\DR^{\rd\infty}(N)$)}. We also have $e^{-c\tau}:\DR_{<c}(N)\isom\DR^{\rd\infty}(\cE^{c\tau}\otimes N)$.

\begin{proposition}\label{prop:RH}
\replacedm{If $N$ as above is a minimal extension at $\tau=0$}{For $N$ as above}, the complexes $\DR_{\leq c}(N)$ and $\DR_{<c}(N)$ have cohomology in degree $0$ at most for each $c\in\CC$, and the Riemann-Hilbert correspondence
\[
N\mto(\cH^0\DR^\an N,\cH^0\DR_\bbullet(N))=(\cF,\cF_\bbullet)
\]
is an equivalence between the full subcategory of the category of holonomic $\Cltau$-modules whose objects are of exponential type at infinity, are minimal extension\addedo{s} \addedo{with a regular singularity} at $0$ and have no other singularity, and the category of minimal Stokes-filtered constructible sheaves on $X$ (\cf Definition \ref{def:consStokes}).

Moreover, under this correspondence, we have
\[
\cF_{<c}=\cH^0\DR_{<c}(N)\quad\forall c\in\CC.
\]
\end{proposition}

\begin{proof}
This is a slight adaptation of the main statement in \cite{Deligne78} (\cf also \cite{Malgrange83bb,B-V89,Malgrange91}).
\end{proof}

\subsection{Sesquilinear pairings}\label{subsec:sesquiN}
Let $\varh:N'\otimes_\CC\ov N{}''\to\cS'(\Afu_\tau)$ be a sesquilinear pairing between holonomic $\Cltau$-modules as considered in the beginning of this section. Then $\varh$ induces a morphism of bicomplexes
\[
\varh_{\DR,0}:\DR^{\rd\infty_\tau}_XN'\otimes_\CC\DR^{\rmod\infty_\tau}_X\ov N{}''\to\Db_X^{\rd\infty_\tau,(\cbbullet,\cbbullet)},
\]
where $\Db_X^{\rd\infty_\tau,(\cbbullet,\cbbullet)}$ is the bicomplex of currents on $X$ with rapid decay along $S^1_{\infty_\tau}$. More generally, since for each $c\in\CC$, the function $e^{\ov{c\tau}-c\tau}$ has moderate growth along $\infty_\tau\in\PP^1_\tau$ or $S^1_{\infty_\tau}\subset X$ as well as all its derivatives, $\varh$ defines a morphism of complexes
\begin{equation}\label{eq:varhc}
\varh_{\DR,c}:\DR_{<c}(N')\otimes_\CC\DR_{\leq -c}(\ov N{}'')\to\Db_X^{\rd\infty,(\cbbullet,\cbbullet)}
\end{equation}
Since the simple complex associated to the double complex $\Db_X^{\rd\infty_\tau,(\cbbullet,\cbbullet)}$ is a resolution of $j_{\infty_\tau,!}\CC_{\Afu_\tau}$, by taking $\cH^0$ we deduce for each~$c$ a pairing
\begin{equation}\label{eq:varhcb}
\varh_{\DR,c}:\cF'_{<c}[1]\otimes\ov\cF{}''_{\leq -c}[1]\to j_{\infty,!}\CC_{\Afu_\tau}[2]
\end{equation}
where $j_{\infty,!}\CC_{\Afu_\tau}[2]$ is the dualizing complex on $X$. All these pairings $\varh_{\rB,c}$ coincide, when restricted to $\Afu$ and using the identification \eqref{eq:identificationc}, with the pairing
\[
\varh_{\DR}:\DR^{\an}N'\otimes_\CC\DR^{\an}\ov N{}''\to\Db_{\Afu}^{(\cbbullet,\cbbullet)}.
\]
In particular, if $\varh:N\otimes_\CC\iota^+\ov N\to\cS'(\Afu_\tau)$ is a $\iota$-sesquilinear pairing on~$N$, it induces a $\iota$-sesquilinear pairing $\varh_{\rB,c}:\cF_{<c}\otimes_\CC\iota^{-1}(\ov\cF_{\leq c})\to j_{\infty,!}\CC_{\Afu_\tau}$ for each $c\in\CC$.

\begin{lemme}\label{lem:RHh}
Let $(\cF',\cF'_\bbullet)$, $(\cF'',\cF''_\bbullet)$ be minimal Stokes-filtered constructible sheaves corresponding to holonomic $\Cltau$-modules $N',N''$ through the equivalence of Proposition \ref{prop:RH}. Then any sesquilinear pairing $\varh_\rB$ between $(\cF',\cF'_\bbullet)$ and $(\cF'',\cF''_\bbullet)$ takes the form $\varh_{\DR}$ for a unique sesquilinear pairing~$\varh$ between $N'$ and $N''$.
\end{lemme}
\addedo{\begin{remarque}
The minimality property at $0$ is assumed in Proposition \ref{prop:RH} and Lemma \ref{lem:RHh} for the sake of simplicity. Without this assumption, the proposition would also hold, but one should first correctly define the category of perverse sheaves on $X$ with a Stokes filtration at infinity. We will not need such a generalization.
\end{remarque}}

\begin{proof}[Proof\addedn{ of Lemma \ref{lem:RHh}}]
The equivalence of categories of Proposition \ref{prop:RH} gives a unique correspondence between morphisms. We will therefore express the pairings as morphisms.

On the one hand, recall (\cf Proposition \ref{prop:duality}) that $(\cF^{\prime\prime\vee},\cF^{\prime\prime\vee}_\bbullet)$ is a minimal Stokes-filtered constructible sheaf, as well as its conjugate, so that $\varh_\rB$ can be regarded as a morphism from $(\cF',\cF'_\bbullet)$ to the conjugate $(\cF^{\prime\prime\dag},\cF^{\prime\prime\dag}_\bbullet)$ of $(\cF^{\prime\prime\vee},\cF^{\prime\prime\vee}_\bbullet)$. By the equivalence of Lemma \ref{lem:equiv} it corresponds in a unique way to a morphism $(\cL',\cL'_\bbullet)\to(\cL^{\prime\prime\dag},\cL^{\prime\prime\dag}_\bbullet)$. The Stokes data of the latter are obtained by conjugating \ref{eq:tStokesdata}.

On the other hand, let us set $N^{\prime\prime\dag}=\Hom_{\ov{\Cltau}}(\ov N{}'',\cS'(\Afu))$, that we consider as a $\Cltau$-module through the $\Cltau$-module structure of $\cS'(\Afu)$. It is known that $N^{\prime\prime\dag}$ is a holonomic $\Cltau$-module which belongs to the category considered in Proposition \ref{prop:RH}. Indeed, this is obtained by sheafifying the construction on $\PP^1$. Then, on $\Afu$, the result follows from \cite{Kashiwara86}, and near $\infty$ it follows from \cite[\S II.3]{Bibi97}. Now, a sesquilinear pairing $\varh:N'\otimes_\CC\ov N{}''\to\cS'(\Afu_\tau)$ is regarded as a $\Cltau$-linear morphism $N'\to N^{\prime\prime\dag}$.

The lemma reduces then to identifying the Stokes data at infinity of $N^{\prime\prime\dag}$ to the conjugate of \ref{eq:tStokesdata} (we will not recall here the classical relationship between Stokes data at infinity for a meromorphic connection $N$ considered as matrices of the form $\id+\text{rapid decay}$ and the Stokes data considered in \S\ref{subsec:Stokesdata}). We will recall in this simple case a sketch of the proof given in \cite[\S II.3]{Bibi97}. We will work locally near infinity, with local coordinate $\hb=1/\tau$, and denote by $\ccN$ the germ of $\cO_{\PP^1}(*\infty)\otimes_{\CC[\tau]}N''$ at infinity. Setting $\cO=\CC\{\hb\}$ and $\cD=\cO\langle\partial_\hb\rangle$, $\ccN$ is a $\cO[\hb^{-1}]$-module with connection and a holonomic $\cD$-module. We also denote by $\Db^{\rmod\infty}$ the germ at $\infty$ of the sheaf $\Db_{\PP^1}^{\rmod\infty}$ already considered in \S\ref{subsec:sesqM} and we set $\ccN^\dag=\cHom_{\ov\cD}(\ov\ccN,\Db^{\rmod\infty})$.

It will be convenient to work on $X$ near $S^1_\infty$. We denote by $\cA^{\rmod\infty}$ the germ along $S^1_\infty$ of $\cA_X^{\rmod\infty}$ and by $\wt\Db{}^{\rmod\infty}$ that of the sheaf on $X$ of distributions having moderate growth along $S^1_\infty$. If $\varpi:X\to\PP^1$ denotes the projection, we set $\wt\ccN=\cA^{\rmod\infty}\otimes_{\varpi^{-1}\cO}\ccN$. There exists $\ccN^\el$ of the form $\ccN^\el=\bigoplus_{i=1}^r(\cE^{\addedo{-}c_i\tau}\otimes\cR_i)$, where each $\cR_i$ has regular singularity, such that $\wt\ccN$ is locally on $S^1_\infty$ isomorphic to $\wt{\ccN^\el}$. It is proved in \loccit that
\begin{itemize}
\item
$(\ccN^\el)^\dag$ is a germ of meromorphic connection at $\infty$,
\item
$\wt{\ccN^\dag}=(\wt\ccN)^\dag\defin\cHom_{\varpi^{-1}\ov\cD}(\varpi^{-1}\ov\ccN,\wt\Db{}^{\rmod\infty})$, which is also a $\bR\cHom$,
\item
$(\wt\ccN)^\dag$ is locally isomorphic to $(\wt{\ccN^\el})^\dag$ and the Stokes data (\ie gluing data) needed to recover $(\wt\ccN)^\dag$ from $(\wt{\ccN^\el})^\dag$ are obtained from those corresponding to $\wt\ccN$ in a natural way, \ie are inverse transposed conjugate of these. (The point is to prove that these inverse transposed conjugate Stokes data are indeed Stokes data, \ie are of the form $\id+\text{rapid decay}$, while they a priori only have moderate growth.)
\end{itemize}
This shows that $\ccN^\dag$ corresponds, via Proposition \ref{prop:RH}\addedo{,} to $(\cL^\dag,\cL^\dag_\bbullet)$.
\end{proof}

\Subsection{Compatibility of the sesquilinear pairing with taking cohomology}
Let $\varh:N\otimes_\CC\iota^+\ov N\to\cS'(\Afu_\tau)$ be a $\iota$-sesquilinear pairing on~$N$, \replacedo{where $N$ is as occurring in the equivalence of categories in Proposition \ref{prop:RH}}{as considered in this section}. Let us fix~$c\in\CC$ satisfying \eqref{assumpt:aa} and \eqref{assumpt:b} of \S\ref{subsec:sesqcohom}.

On the one hand, $\varh$ defines $\varh_{\DR,c}:\cF_{\leq c}\otimes_\CC\iota^{-1}(\ov\cF_{\leq c})\to j_{\infty,!}\CC_{\Afu_\tau}$ (because $\cF_{<c}=\nobreak\cF_{\leq c}$), and then $\wh{\varh_{\DR,c}}$ on $H^1(X,\cF_{\leq c})$, according to Proposition \ref{prop:pairing}.

On the other hand, let us denote by $M=\Foub N$ the inverse Laplace transform of~$N$, which by assumption is a regular holonomic $\Clt$-module with singular set contained in~$C$. The complex $M\To{t-c}M$ is identified with the algebraic de~Rham complex $N\To{-\partial_\tau-c}N$, which can be computed analytically as $\bR\Gamma\big(X,\DR^{\rmod\infty}(\cE^{c\tau}\otimes\nobreak N)\big)$, and we have seen that this complex is isomorphic to $\bR\Gamma\big(X,\DR_{\leq c}(N)\big)$ via the multiplication by $e^{c\tau}$ termwise. For $c\not\in C$, the latter complex has cohomology in degree $1$ at most, so the fibre $M/(t-c)M$ is identified with $H^1(X,\cF_{\leq c})$. Now, the inverse Fourier transform $\ov F_\tau \varh$ is a sesquilinear pairing on the inverse Laplace transform~$M$ of~$N$. We set $h=\wh \varh=-2\pi i \ov F_\tau \varh$. Restricting it to $\Afu_t\moins C$, it takes values in $\cC^\infty_{\Afu_t\moins C}$. Restricting it to the fibre at~$c$ also induces a sesquilinear pairing on $H^1(X,\cF_{\leq c})$, that we denote by $h_{\DR,c}=\wh \varh_{\DR,c}$. We will give a detailed proof of the following lemma in the appendix.

\begin{lemme}\label{lem:hBc}
We have $\wh \varh_{\DR,c}=\wh{\varh_{\DR,c}}$.
\end{lemme}

\subsection{The main theorem}
Let \replacedo{$((G_{c,1},G_{c,2})_{c\in C},S,S')$}{$(C,S=(S_{ij})_{i,j=1,\dots, n},S'=(S'_{ij})_{i,j=1,\dots, n})$} be Stokes data \addedo{of type $(C,\theta_o)$} as in Definition \ref{def:catStokesdata}\deletedm{, which satisfy the minimality property of Definition \ref{def:critminext}}. Let us fix bases of $G_{c,\ell}$, $c\in C$ and $\ell=1,2$ and let us denote by $\Sigma,\Sigma'$ the matrices of $S,S'$ in these bases.\addedm{ The choice of bases also fixes a sesquilinear form $\varh_{1\ov2}$ whose matrix in these bases is the identity.} According to the Riemann-Hilbert correspondence, these data define a meromorphic bundle $(\cH(*0),\nabla)$ on $\Afuan_\hb$ with connection \replacedo{having a}{with} pole at $\hb=0$ only. The connection is of exponential type. We denote by $(\cH,\nabla)$ the Deligne-Malgrange lattice \deletedm{of} $\addedm{\DM^{>0}}(\cH(*0),\nabla)$.

If $\Sigma'=-{}^t\ov\Sigma$, the local system attached to $(\cH(*0),\nabla)$ is equipped with a $\iota$\nobreakdash-skew-Hermitian pairing $\varh_\rB$, and we can apply the ``twistor gluing procedure'' of \addedn{\cite[Lemma 2.14]{Hertling01}} and \cite[Def\ptbl1.25]{Bibi05} (\cf\S\ref{subsec:tw})\addedm{ by using the $\iota$-Hermitian pairing $-2\pi i\varh_\rB$}.\addedo{\deletedm{ In such a case, let us also make explicit the minimality assumption on the Stokes data (see also Remark \ref{rem:minassumption}): it means that the radical of the Hermitian form $\Sigma+{}^t\ov\Sigma$ contains no non-zero vector in one of the blocks $G_{c,1}$ ($c\in C$).}}

\begin{theoreme}\label{th:main}
\addedn{Let $((G_{c,1},G_{c,2})_{c\in C},S,S')$ be Stokes data of type $(C,\theta_o)$\deletedm{ which satisfy the minimality property}.} Assume that there exist bases of $G_{c,\ell}$, $c\in C$ and $\ell=1,2$ such that the matrices $\Sigma,\Sigma'$ of $S,S'$ satisfy
\begin{align*}\tag*{\protect(\ref{th:main})($*$)}\label{eq:main*}
&\Sigma'=-{}^t\ov\Sigma\quad\text{and}\quad\Sigma+{}^t\ov\Sigma \text{ is positive semi-definite,}\\
\tag*{\protect(\ref{th:main})($**$)}\label{eq:main**}&\addedm{\forall c\in C,\quad \text{either $K_c=0$ or $2\pi i\varh_{K_c}$ is positive definite on $K_c$.}}
\end{align*}
\addedm{(\cf Remark \ref{rem:Kc} for $\varh_{K_c}$.)} Then, the twistor structure on $\PP^1$ obtained from $(\cH,\nabla,-2\pi i\varh_\rB)$, (where $(\cH,\nabla)$ is the Deligne-Malgrange lattice defined by $\Sigma$) by the ``twistor gluing procedure'' is pure of weight $0$ and polarized.
\end{theoreme}

\begin{proof}
\addedm{Condition \ref{eq:main**} implies that $\varh_{K_c}$ is nondegenerate for each $c\in C$. Then, as we already noticed at the end of Remark \ref{rem:Kc}, the Stokes data enriched with the sesquilinear form $\varh_{1\ov2}$ split as a direct sum of minimal Stokes data and trivial Stokes data on each $K_c$ (both enriched with sesquilinear forms). The proof splits correspondingly. The non-trivial part concerns the case of minimal Stokes data (all $K_c$ equal to zero), that we consider now.}

In accordance with the previous part of the article, we will work with the variable $\tau=1/\hb$\addedo{, so we now denote by $(\cH(*\infty),\nabla)$ the meromorphic bundle defined above on $\PP^1_\tau\moins\{0\}$}. Let us denote by $\ccN(*0)$ the Deligne meromorphic extension \addedo{(with a regular singularity)} at $\tau=0$ of $\cH(*\infty)$, by $\ccN$ its minimal extension at $\tau=0$ and by $N$ \replacedo{the}{its} global sections \addedo{of $\ccN$} on $\PP^1$.

As indicated in Lemma \ref{lem:varhs}, the pairing $\varh_{\rB,\infty}$ on $(\cL,\cL_\bbullet)$ determined by the Stokes data and the properties of $\Sigma,\Sigma'$ (\cf\S\ref{subsec:htheta}) gives rise in a unique way to a $\iota$-sesquilinear form $\varh_\rB$ on $(\cF,\cF_\bbullet)$ which restricts (in the sense of Proposition \ref{prop:duality}) to $\varh_{\rB,\infty}$ on $S^1_\infty$, and, according to Lemma \ref{lem:RHh}, to a unique $\iota$-sesquilinear pairing $\varh$ on $N$. Let us choose $c\in\CC$ which satisfies both properties \eqref{assumpt:aa} and \eqref{assumpt:b} of \S\ref{subsec:sesqcohom} with respect to $\cF$ (\eqref{assumpt:aa} means that $c\not\in C$). Then, after our assumption on $\Sigma,\Sigma'$ and according to Corollary \ref{cor:posdef} and Proposition \ref{prop:pairing}, $\wh{\varh_{\DR,c}}$ is positive definite. Hence, $\wh\varh_{\DR,c}$ is so, according to Lemma \ref{lem:hBc}. We conclude by applying Corollary \ref{cor:N}\addedm{, since $\Fou\,\wh\varh=-2\pi i\varh$}.

\addedm{%
Let us now consider the easy case with enriched Stokes data $(K_{c_o},S(K_{c_o}),S,S,\varh_{K_{c_o}})$, which are isomorphic to enriched Stokes data $(K_{c_o},K_{c_o},\id,\id,\varh_{K_{c_o}})$, where $\varh_{K_{c_o}}$ is a nondegenerate skew-Hermitian form on $K_{c_o}$. We will first adapt Proposition \ref{prop:RH} and Lemma \ref{lem:RHh} in this case.}

\addedm{
Let us set $N=(K_{c_o}\otimes\CC[\tau],d-c_o\id d\tau)$ and consider some Hermitian form $h_{K_{c_o}}$ on $K_{c_o}$. Then we define the $\iota$-Hermitian form $\varh$ on $N$ by the formula
\[
\varh(u_{c_o}\otimes f(\tau),\ov{v_{c_o}\otimes g(\tau)})=h_{K_{c_o}}(u_{c_o},\ov v_{c_o})f(\tau)\ov{g(-\tau)}e^{\ov{c_o\tau}-c_o\tau}.
\]
The complex $\DR^\an N$ has cohomology in degree $0$ only and $\cH^0\DR^\an N\!\subset\!\cO_{\Afuan_\tau}\otimes\nobreak N=K_{c_o}\otimes \cO_{\Afuan_\tau}$ is identified with the constant sheaf $K_{c_o}\otimes \CC\cdot e^{c_o\tau}$ and, with this identification,
\[
\varh_{\DR}(u_{c_o}\otimes e^{c_o\tau},\ov{v_{c_o}\otimes e^{-c_o\tau}})=h_{K_{c_o}}(u_{c_o},\ov v_{c_o}).
\]
Setting $G=(K_{c_o}\otimes\CC[\tau,\tau^{-1}],d-c_od\tau)=(K_{c_o}\otimes\CC[\hb,\hb^{-1}],d+c_od\hb/\hb^2)$, we have $\DM^0G=K_{c_o}\otimes\nobreak\CC[\hb]$ and $\DM^{>0}G=K_{c_o}\otimes\nobreak(\hb\CC[\hb])$. It is easy to check (\cf \cite[Example 1.33(1)]{Bibi05}) that, if $h_{K_{c_o}}$ is positive definite on $K_{c_o}$, then $(\DM^0G,\nabla,\varh_{\DR})$ defines, by twistor gluing, an integrable twistor structure which is pure of weight $0$ and polarized. It easily follows that $(\DM^{>0}G,\nabla,-\varh_{\DR})$ defines, by twistor gluing, an integrable twistor structure which is pure of weight $0$ and polarized. Below, we will not distinguish between $\varh_{\DR}$ and $h_{K_{c_o}}$.}

\addedm{In conclusion, starting with the skew-Hermitian form $\varh_{K_{c_o}}$ on $K_{c_o}$, if the Hermitian form $h_{K_{c_o}}\defin2\pi i\varh_{K_{c_o}}$ is positive definite, then $(\DM^{>0}G,\nabla,-2\pi i\varh_{K_{c_o}})$ defines, by twistor gluing, an integrable twistor structure which is pure of weight $0$ and polarized.}
\end{proof}

\begin{remarques}\mbox{}
\begin{enumerate}
\item
This statement was conjectured (and proved in a particular case) in \cite[Conj\ptbl10.2]{H-S06} and was the main motivation for proving Theorem \ref{th:main}\addedo{.}\deletedo{, so that one can apply the results of \cite{Bibi05},} As in the partic\-ular case treated in \cite[Lemma\ptbl10.1]{H-S06}\addedo{, the main idea is to apply the results of~\cite{Bibi05}}.
\item
If $\Sigma$ is \emph{real}, the integrable twistor structure that we get is a TERP structure in the sense of \cite{Hertling01}. If $\Sigma$ is rational, we get a non-commutative Hodge structure, in the sense of \cite{K-K-P08}.
\item
The simplest example of a complex variation of polarized Hodge structure on $\Afu\moins C$ is that of a holomorphic vector bundle $V$ with a flat Hermitian metric (the weight is zero and the Hodge type is $(0,0)$). Recall more generally that variations of polarized Hodge structures on $\Afu\moins C$ correspond exactly to variations of polarized pure integrable twistor structures which are tame (\ie have regular singularities) at the singularities $C\cup\{\infty\}$, after \cite[Th\ptbl 6.2]{H-S08}.

Similarly, Theorem \ref{th:main} gives the simplest example of an integrable twistor structure whose associated variation by rescaling $\hb$ (\cf \cite[\S4]{H-S06} and \cite[\S2.d]{Bibi05}) is wild at $\tau=\infty$. It is obtained by Fourier-Laplace transform\addedo{ation} from the previous one.

\addedo{\deletedm{Theorem \ref{th:main} may be not an optimal statement, and one can expect that part of the minimality assumption can be relaxed.}}

\addedo{\item
One can conjecture a kind of converse of Theorem \ref{th:main} in the following way. Given a block upper triangular matrix $\Sigma$ such that the diagonal blocks $\Sigma_{ii}$ (\hbox{$i=1,\dots,n$}) are invertible, then $\Sigma+{}^t\ov\Sigma$ is positive semi-definite if, for all pairs $(C,\theta_o)$ consisting of a subset $C\subset\CC$ with $\#C=n$ and $\theta_o\in\RR/2\pi\ZZ$ generic with respect to $C$ (\cf \S\ref{subsec:Stokesdata}), the corresponding twistor structure considered in Theorem \ref{th:main} is pure and polarized.}
\end{enumerate}
\end{remarques}

\section*{Appendix: Proof of Lemma \ref{lem:hBc}}
\setcounter{section}{1}\setcounter{subsection}{0}\setcounter{equation}{0}
\def\thesection{\Alph{section}}

We will give a detailed proof of this lemma. A proof of a similar result had only been sketched for \cite[Prop\ptbl1.18]{Bibi05}. We will keep the setting of \S\ref{sec:RHsesqui}.

\subsection{Integral formula for the Fourier transform of a sesquilinear pairing between $\cD$-modules}
We start by expressing~$\wh\varh$ (as defined in Corollary \ref{cor:N}\eqref{cor:N2}) by an integral formula.

Let us first recall that the inverse \replacedo{Laplace}{Fourier} transform~$M$ of $N$ can be obtained as the algebraic direct image $q_+(p^+N\otimes_{\CC[t,\tau]}E^{t\tau})$, where $p$ (\resp $q$) denotes the projection from $\Afu_\tau\times\Afu_t$ to $\Afu_\tau$ (\resp~$\Afu_t$) and $E^{t\tau}=(\CC[t,\tau],d+\tau dt+td\tau)$. Moreover, this formula can be sheafified and made analytic, giving $\cM=q_+(p^+\cN\otimes\cE^{t\tau})$, where $\cN$ is as in the beginning of Section \ref{sec:RHsesqui}, and $p,q$ now denote the projections $\PP^1_\tau\times\PP^1_t\to\PP^1_\tau$ or $\PP^1_t$ in the analytic category (\cf \cite[Appendix~A]{D-S02a} for details). Since we are only interested in the behaviour on $\Afu_t\moins C$, we will set $Y=\Afu_t\moins C$ and denote by $p:\PP^1_\tau\times Y\to \PP^1_\tau$ the projection, and similarly for $q$. Then $\cM$ is a $\cD_Y$-module. By assumption on~$C$, it is $\cO_Y$-locally free. More precisely, if $\cN$ has the connection~$\nabla$, $p^+\cN\otimes\cE^{t\tau}$ is $p^*\cN\defin\cO_{\PP^1_\tau\times Y}\otimes_{p^{-1}\cO_{\PP^1_\tau}}p^{-1}\cN$ equipped with the connection $p^*\nabla+\tau dt+td\tau$, and~$\cM$ is the first cohomology of the relative de~Rham complex
\begin{equation}\label{eq:DRrel}
q_*p^*\cN\To{\nabla+td\tau}q_*(p^*\cN\otimes\Omega^1_{\PP^1_\tau\times Y/Y})
\end{equation}
equipped with the connection induced by $d_Y+\tau dt$, where $d_Y$ is the differential relative to $Y$. Notice that $p^*\cN$ and $p^*\cN\otimes\Omega^1_{\PP^1_\tau\times Y/Y}$ are $q_*$-acyclic (\cf \loccit). Moreover this complex has cohomology in degree one at most.

For the sake of simplicity, we will denote the volume form $\itwopi d\tau\wedge d\ov\tau$ (on $\Afu_\tau$) by $d\vol_\tau$, its $t$-analogue (on $Y$) by $d\vol_t$.

We will use the following lemma:
\begin{lemme}\label{lem:rapiddecay}
Let $\varphi$ be $C^\infty$ on $\PP^1_\tau\times Y$, with compact support. Then the function $\tau\mto\int_Ye^{t\tau-\ov{t\tau}}\varphi \,d\vol_t$ is $C^\infty$ on $\Afu_\tau$, with rapid decay at infinity as well as all its derivatives.
\end{lemme}

\begin{proof}
By assumption, $\varphi$ induces a $C^\infty$ function \addedo{on} $\Afu_\tau\times Y$ such that $\module{t}^{\addedo{m}}\partial_\tau^\alpha\partial_{\ov \tau}^\beta\partial_t^\gamma\partial_{\ov t}^\delta\varphi$ is bounded for all $\alpha,\beta,\gamma,\delta,\addedo{m}\geq0$ (since, in the coordinate $\tau'=1/\tau$, we have $\partial_\tau\varphi=-\tau^{\prime2}\partial_{\tau'}\varphi$ and similarly for $\partial_{\ov\tau}$).

Let us still denote by $\ov F_t\varphi$ the integral we consider (with some abuse, since $\varphi$ depends on $\tau$). That $\ov F_t\varphi$ is $C^\infty$ on $\Afu_\tau$ is clear. It is a matter of showing that each expression $\big\Vert \tau^a\ov\tau{}^b\partial_\tau^c\partial_{\ov \tau}^d\ov F_t(\varphi)\big\Vert_{L^\infty}$ is bounded. We have
\[
\tau^a\ov\tau{}^b\partial_\tau^c\partial_{\ov \tau}^d\ov F_t(\varphi)=\ov F_t\big((-\partial_t)^a\partial_{\ov t}^b(t+\partial_\tau)^c(-\ov t+\partial_{\ov\tau})^d\varphi\big).
\]
Since $\psi\defin(-\partial_t)^a\partial_{\ov t}^b(t+\partial_\tau)^c(-\ov t+\partial_{\ov\tau})^d\varphi$ satisfies the same properties as $\varphi$ does, as indicated above, it is enough to get a bound for $F_t\psi$ for such a $\psi$. Since $|e^{t\tau-\ov{t\tau}}|=1$, we have, for $\addedo{m}$ such that $(1+|t|^{2\addedo{m}})^{-1}$ is $L^1$ on~$\Afu_t$,
\[
\norme{\ov F_t\psi}_{L^\infty}\leq\norme{(1+|t|^{2\addedo{m}})^{-1}}_{L^1}\norme{(1+|t|^{2\addedo{m}})\psi}_{L^\infty}<+\infty.\qedhere
\]
\end{proof}

Let us still denote by $\varh$ the sheafified sesquilinear pairing $\cN\otimes_\CC\iota^+\ov\cN\to\Db_{\PP^1_\tau}^{\rmod\infty_\tau}$ (where $\cN$ is as in the beginning of Section \ref{sec:RHsesqui}). We will define a sesquilinear pairing
\[
\varh':(p^+\cN\otimes\cE^{t\tau})\otimes \iota^+\ov{p^+\cN\otimes\cE^{t\tau}}\to\Db_{\PP^1_\tau\times Y}.
\]
We will identify $\iota^+\ov{p^+\cN\otimes\cE^{t\tau}}$ with $\ov{(p^+\iota^+\cN)\otimes\cE^{-t\tau}}$, where $\iota$ denotes the involution $\tau\mto-\tau$ either before or after $p^+$. Let $n$ (\resp $n'$) be a local section of $p^+\cN$ (\resp $p^+\iota^+\cN$). We can write $n=\sum_i\phi_in_i$, $n'=\sum_j\psi_jn'_j$ (finite sums) where $n_i$ (\resp $n'_j$) are local sections of $\cN$ (\resp $\iota^+\cN$) and $\phi_i,\psi_j$ are local sections of $\cO_{\PP^1_\tau\times Y}$.

Let $\varphi$ be a $C^\infty$ form with compact support of maximal degree \addedo{(namely, degree~$4$)} on an open subset of $\PP^1_\tau\times Y$ where $n,n'$ are defined. By Lemma \ref{lem:rapiddecay}, the $2$-form $\tau\mto\int_Ye^{t\tau-\ov{t\tau}}\phi_i\psi_j\varphi$ has rapid decay at~$\infty_\tau$ for every $i,j$, so we can set \deletedo{set}
\begin{equation}\label{eq:varh'}
\langle\varh'(n,\ov n'),\varphi\rangle=\sum_{ij}\Big\langle\varh(n_i,\ov n'_j),\Big(\int_Ye^{t\tau-\ov{t\tau}}\phi_i\ov\psi_j\varphi\Big)\Big\rangle,
\end{equation}
and one checks that this does not depend on the chosen decomposition of $n,n'$, so that $\varh'(n,\ov n')$ is a section of $\Db_{\PP^1_\tau\times Y}$.

Integration of currents along $\PP^1_\tau$ composed with $\varh'$ induces a sesquilinear pairing
\[
q_*(p^*\cN\otimes\Omega^1_{\PP^1_\tau\times Y/Y})\otimes q_*\iota^*(\ov{p^*\cN\otimes\Omega^1_{\PP^1_\tau\times Y/Y}})\to\Db_Y
\]
which becomes a sesquilinear pairing (noticing that $\iota^*$ disappears after $q_*$ and using the twisted differentials as in \eqref{eq:DRrel})
\begin{equation}
q_+\varh':\cM\otimes\ov\cM\to\Db_Y.
\end{equation}
(That $q_+\varh'$ is well-defined and sesquilinear is checked in a standard way.)

\begin{lemme}\label{lem:compvarh}
We have \replacedo{$\wh\varh= q_+\varh'$}{$\itwopi\wh\varh= q_+\varh'$}.
\end{lemme}

\begin{proof}
Since $\cM$ is generated by~$M$ (inverse Laplace transform of $N$), and since $M=N$ as $\CC$-vector spaces, it is enough to check the equality for the values at $n,n'\in N$, according to sesquilinearity. Let $n,n'\in N$ and let $\eta$ be a $C^\infty$ $2$-form on $Y$ with compact support. By definition,
\[
\langle\addedo{\wh\varh(n,\ov n')},\eta\rangle=\langle\addedo{-2\pi i}\ov F_\tau\varh(n,\ov n'),\eta\rangle=\langle\varh(n,\ov n'),(\ov F_t\eta)d\tau\wedge d\ov\tau\rangle
\]
where $\ov F_t\eta=\int_Ye^{t\tau-\ov{t\tau}}\eta$. On the other hand, if we set $N[t]=\CC[t]\otimes_\CC N$, according to the identification \hbox{$M=N[t]\cdot d\tau/(p^*\nabla+td\tau)N[t]$}, we have
\[
\langle q_+\varh'(n,\ov n'),\eta\rangle=\langle \varh'(nd\tau,\ov{n'd\tau}),q^*\eta\rangle=\Big\langle\varh(n,\ov n'),\Big(\int_Ye^{t\tau-\ov{t\tau}}\addedo{\eta}\Big)d\tau\wedge d\ov\tau\Big\rangle,
\]
hence the assertion.
\end{proof}

\subsection{Lifting to $\wt\PP^1_\tau\times Y$ and restriction to $t=c$}
Let $\varpi:\wt\PP^1_\tau\times Y\to\PP^1_\tau\times Y$ be the oriented real blow up of $\infty_\tau\times Y$, and let $\cA_{\wt\PP^1_\tau\times Y}^{\rmod\infty_\tau},\cA_{\wt\PP^1_\tau\times Y}^{\rd\infty_\tau}$ be the corresponding sheaf of holomorphic functions which have moderate growth (\resp rapid decay) along $S^1_{\infty_\tau}\times Y$.

We will lift the sesquilinear pairing $\varh'$ as a sesquilinear pairing
\[
\wt\varh':\Big[\cA_{\wt\PP^1_\tau\times Y}^{\rd\infty_\tau}\otimes\varpi^{-1}(p^+\cN\otimes\cE^{t\tau})\Big]\otimes\Big[\ov{\cA_{\wt\PP^1_\tau\times Y}^{\rmod\infty_\tau}\otimes\varpi^{-1}\iota^+(p^+\cN\otimes\cE^{t\tau})}\Big]\to\Db_{\wt\PP^1_\tau\times Y}^{\rd\infty_\tau}.
\]
In order to do so, we use a formula analogous to \eqref{eq:varh'}, where now each $\phi_i$ is a local section of $\cA_{\wt\PP^1_\tau\times Y}^{\rd\infty_\tau}$, each $\psi_j$ is a local section of $\cA_{\wt\PP^1_\tau\times Y}^{\rmod\infty_\tau}$, and $\varphi$ has moderate growth along $S^1_{\infty_\tau}\times Y$. Since each term $\phi_i\ov\psi_j\varphi$ has rapid decay along $S^1_{\infty_\tau}\times Y$, \replacedo{the f}{F}ormula \eqref{eq:varh'} remains meaningful.

We will denote by $\Db_{\wt\PP^1_\tau\times Y/Y}^{\rd\infty_\tau}$ the subsheaf of $\Db_{\wt\PP^1_\tau\times Y}^{\rd\infty_\tau}$ consisting of distributions which are $C^\infty$ with respect to $t\in Y$. For such a distribution, the evaluation at $t=c\in Y$ is well defined as a distribution on $\wt\PP^1_\tau$ with rapid decay at $\tau=\infty$.

Recall (\cf \replacedo{after the proof of Lemma \ref{lem:rapiddecay}}{\S\ref{subsec:sesquiN}}) that the sesquilinear pairing $\varh:N\otimes_\CC\iota^+\ov N\to\cS'(\Afu_\tau)$ can be sheafified and lifted as above as a sesquilinear pairing
\[
\wt\varh_c:\Big[\cA_{\wt\PP^1_\tau}^{\rd\infty_\tau}\otimes\varpi^{-1}(\cN\otimes\cE^{c\tau})\Big]\otimes\Big[\ov{\cA_{\wt\PP^1_\tau}^{\rmod\infty_\tau}\otimes\varpi^{-1}\iota^+(\cN\otimes\cE^{c\tau})}\Big]\to\Db_{\wt\PP^1_\tau}^{\rd\infty_\tau}.
\]

\begin{proposition}
The sesquilinear pairing $\wt\varh'$ takes values in $\Db_{\wt\PP^1_\tau\times Y/Y}^{\rd\infty_\tau}$ and, for each $c\in Y$, its evaluation at $t=c$ is equal to $\wt\varh_c$.
\end{proposition}

\begin{proof}
The second part of the statement is clear once we have shown the first part, that we consider now. Notice first that it is a local statement.

Firstly, on $\Afu_\tau\times Y$ the statement is clear, since $e^{\ov{t\tau}-t\tau}\wt\varh'$, when expressed on sections $n$ of $\cN$ and $n'$ of $\iota^+\cN$, takes values in distributions annihilated by $\partial_t$ and $\ov\partial_t$.

We will thus consider the statement locally near $S^1_{\infty_\tau}\times Y$ and \replacedo{we}{will} will show that, locally near $(\addedo{e^{i\theta_o}},c)\in S^1_{\infty_\tau}\times Y$, $\wt\varh'$ takes values in $C^\infty$ functions which are infinitely flat along $S^1_{\infty_\tau}\times Y$, by analyzing the differential equations satisfied by $\wt\varh'(n,\ov n')$. By using the Hukuhara-Turrittin theorem for $\cN$ at $\tau=\infty$ (\cf\eg\cite[Appendix]{Malgrange91}), we are reduced to evaluating $\wt\varh'$ on sections $n,n'$ which are solutions of
\[
(\tau\partial_\tau+c_i\tau+\alpha)^{\addedo{m}}n=0,\quad(\tau\partial_\tau+c_j\tau+\alpha')^{\addedo{m}}n'=0,\qquad c_i,c_j\in C,\ \alpha,\alpha'\in\CC,\ \addedo{m}\gg0.
\]
When restricted to $\tau\neq\infty_\tau$, $\wt\varh(n,\ov n')\defin\wt\varh_0(n,\ov n')$ is a $C^\infty$ function and the function $\wt\varh\big((e^{c_i\tau}\tau^\alpha n,\ov{(e^{c_j\tau}\tau^{\alpha'}n'})\big)$ is annihilated by $(\tau\partial_\tau)^{\addedo{m}}$ and $(\ov\tau\ov\partial_\tau)^{\addedo{m}}$.

Similarly, $e^{\addedo{\ov{t\tau}-t\tau}}\wt\varh'\big((e^{c_i\tau}\tau^\alpha n,\ov{(e^{c_j\tau}\tau^{\alpha'}n'})\big)$ is annihilated by $(\tau\partial_\tau)^{\addedo{m}}$, $(\ov\tau\ov\partial_\tau)^{\addedo{m}}$, $\partial_t$ and~$\ov\partial_t$. Therefore, this function does not depend on $t$ and has moderate growth in some neighbourhood of $(\addedo{e^{i\theta_o}},c)$.

If $\reel((c-c_i)e^{i\theta_o}-\ov{(c-c_j)e^{i\theta_o}})<0$, then in some neighbourhood of $(\addedo{e^{i\theta_o}},c)$ we have $\reel((t-c_i)\tau-\ov{(t-c_j)\tau})<0$, and $\wt\varh'(n,\ov n')$ is a $C^\infty$ function with rapid decay along $S^1_{\infty_\tau}\times Y$ on this neighbourhood, as wanted.

Otherwise, there is an open set of $\Afu_\tau\times Y$ containing $(\addedo{e^{i\theta_o}},c)$ in its closure, on which $\reel((t-c_i)\tau-\ov{(t-c_j)\tau})>0$, hence on which the function $\wt\varh'(n,\ov n')$ cannot be extended as a distribution (and even a moderate distribution), unless it is identically zero. In such a case, the desired statement trivially holds.
\end{proof}

\subsection{The de~Rham complexes}
In the neighbourhood of $\infty_\tau\times Y$, $p^+\cN\otimes\cE^{t\tau}$ is a meromorphic bundle with connection. Since $t$ varies in $Y=\Afu_t\moins C$, this meromorphic bundle with connection is good, in the sense of \cite[\S3.2]{Malgrange95}. It follows, by an extension of the result in dimension one, proved in \cite[Appendix~1]{Malgrange91} for instance, that the moderate and rapid decay de~Rham complexes $\DR^{\rd\infty_\tau}(p^+\cN\otimes\cE^{t\tau})$ and $\DR^{\rmod\infty_\tau}(p^+\cN\otimes\cE^{t\tau})$ have cohomology in degree zero at most in the neighbourhood of $\infty_\tau\times Y$. On the other hand, on $\Afu_\tau\times Y$, $p^+\cN\otimes\cE^{t\tau}$ is isomorphic to $p^+\cN$, and since $\DR^\an\cN$ has cohomology in degree $0$ at most (because $\cN$ is assumed to be a minimal extension at $\tau=0$), so does $\DR^\an p^+\cN$. Therefore, both complexes $\DR^{\rd\infty_\tau}(p^+\cN\otimes\cE^{t\tau})$ and $\DR^{\rmod\infty_\tau}(p^+\cN\otimes\cE^{t\tau})$ have cohomology in degree zero at most.

Recall also (\cf Proposition \ref{prop:RH}) that, for each $c\in Y$, the complexes $\DR_{<c}N$ and $\DR_{\leq c}N$ have cohomology in degree $0$ at most.

\begin{proposition}\label{prop:evalc}
For each $c\in Y$, there are functorial morphisms
\begin{align*}
i_c^{-1}\cH^0(\DR^{\rd\infty_\tau}(p^+\cN\otimes\cE^{t\tau}))&\to\cH^0(\DR_{<c}N)\\
\tag*{and}
i_c^{-1}\cH^0(\DR^{\rmod\infty_\tau}(p^+\cN\otimes\cE^{t\tau}))&\to\cH^0(\DR_{\leq c}N)
\end{align*}
which are isomorphisms, where $i_c:\wt\PP^1_\tau\times\{c\}\hto\wt\PP^1_\tau\times Y$ denotes the inclusion.
\end{proposition}

\begin{proof}
We will show the proposition in the moderate case, the rapid decay case being similar. Let us denote by $\DR^{\rmod\infty_\tau}_{\text{rel}}$ the relative de~Rham complex (with differential forms relative to the projection $\wt\PP^1_\tau\times Y\to Y$ only). Evaluating the coefficients at $t=c$ induces a natural morphism of complexes $i_c^{-1}\DR^{\rmod\infty_\tau}_{\text{rel}}(p^+\cN\otimes\cE^{t\tau})\to\DR_{\leq c}N$, since the evaluation at $t=c$ of a section of $e^{t\tau}\cA^{\rmod\infty_\tau}_{\wt\PP^1_\tau\times Y}$ belongs to $e^{c\tau}\cA^{\rmod\infty_\tau}_{\wt\PP^1_\tau}$.

We have a natural action of $\partial_t$ on $i_c^{-1}\DR^{\rmod\infty_\tau}_{\text{rel}}(p^+\cN\otimes\cE^{t\tau})$, and the kernel of~$\partial_t$ on its $\cH^0$ is equal to $i_c^{-1}\cH^0\DR^{\rmod\infty_\tau}(p^+\cN\otimes\cE^{t\tau})$, hence we obtain the desired morphism to $\cH^0(\DR_{\leq c}N)$. To show that this morphism is an isomorphism is now a local question on $\wt\PP^1_\tau\times Y$.

On $\Afu_\tau\times Y$, the result is clear since $p^+\cN\otimes\cE^{t\tau}\simeq p^+\cN$. Near $(\addedo{e^{i\theta_o}},c)\in S^1_{\infty_\tau}\times Y$, we can reduce, by functoriality, that $\cN=\cN_\alpha\otimes\cE^{\addedo{-}c_i\tau}$, with $c_i\in C$ and $\cN_\alpha=(\cO_{\PP^1_\tau}(*\infty_\tau),d+\alpha d\tau/\tau)$. One is reduced to showing that $e^{(t-c_i)\tau}$ has moderate growth along $S^1_{\infty_\tau}\times Y$ near $(\addedo{e^{i\theta_o}},c)$ if and only if $e^{(c-c_i)\tau}$ has moderate growth along $S^1_{\infty_\tau}$ near $\addedo{e^{i\theta_o}}$. This is clear, and also equivalent to the fact that both functions have rapid decay, since $c\neq c_i$.
\end{proof}

From $\wt\varh'$ we derive $\wt\varh'_{\DR}$, which is a pairing of double complexes
\[
\wt\varh'_{\DR}:\DR^{\rd\infty_\tau}(p^+\cN\otimes\cE^{t\tau})\otimes\ov{\DR^{\rmod\infty_\tau}\iota^+(p^+\cN\otimes\cE^{t\tau})}\to\Db_{\wt\PP^1_\tau\times Y/Y}^{\rd\infty_\tau,(\cbbullet,\cbbullet)}.
\]
The de~Rham complex of currents which are $C^\infty$ with respect to $Y$ and have rapid decay along $S^1_{\infty_\tau}\times Y$, which is the simple complex associate to the Dolbeault complex $\Db_{\wt\PP^1_\tau\times Y/Y}^{\rd\infty_\tau,(\cbbullet,\cbbullet)}$, is a resolution of $j_{\infty,!}\CC_{\Afu_\tau\times Y}$.

By Proposition \ref{prop:evalc}, applying the evaluation at $t=c$ to $\wt\varh'_{\DR}$ defines a pairing $\wt\varh'_{\DR,c}$ which is nothing but the pairing $\varh_{\DR,c}$ considered in \eqref{eq:varhc}.

\subsection{End of the proof of Lemma \ref{lem:hBc}}
From Lemma \ref{lem:compvarh} we deduce that $\addedo{\wh\varh_{\DR,c}}=(q_+\varh')_{\DR,c}$. Denoting by $\wt q$ the projection $\wt\PP^1_\tau\times Y\to Y$, we also obtain
\[
\addedo{\wh\varh_{\DR,c}}=(\wt q_+\wt\varh')_{\DR,c},
\]
where $\wt q_+$ denotes the integration of currents along the fibres of $\wt q$. Integration of currents of any degree is compatible with the differentials, therefore $(\wt q_+\wt\varh')_{\DR,c}=(\wt q_+\wt\varh'_{\DR})_c$. On the other hand, evaluation at $t=c$ is compatible with the integration of currents which are $C^\infty$ with respect to $Y$, so we finally get
\[
\addedo{\wh\varh_{\DR,c}}=(\wt q_+\wt\varh'_{\DR,c}).
\]
On the other hand, we have seen above that $\wt\varh'_{\DR,c}$ is identified with $\varh_{\DR,c}$, and its integration $\wt q_+\wt\varh'_{\DR,c}$ is nothing but the pairing induced on the cohomology, that is, $\wh{\varh_{\DR,c}}$.\qed

\backmatter

\end{document}